\documentclass[11pt]{article}
\usepackage[a4paper, total={6in, 8in}]{geometry}
\usepackage{enumerate}
\usepackage{amsmath}
\usepackage{amssymb,latexsym}
\usepackage{amsthm}
\usepackage{color}
\usepackage{cancel}
\usepackage{graphicx}
\usepackage{cite}
\usepackage{changes}
\usepackage{lineno}
\usepackage{hyperref}
\usepackage{comment}
\usepackage{amscd}
\usepackage{cleveref}
\usepackage{enumitem}

\DeclareMathOperator{\C}{\mathcal{C}}

\newtheorem{theorem}{Theorem}[section]

\newtheorem{lemma}[theorem]{Lemma}
\newtheorem{corollary}[theorem]{Corollary}
\newtheorem{definition}[theorem]{Definition}
\newtheorem{proposition}[theorem]{Proposition}

\newtheorem{remark}[theorem]{Remark}

\newcommand{\cA}{{\mathcal A}}

\newcommand{\cC}{{\mathcal C}}

\newcommand{\cN}{{\mathcal N}}

\newcommand{\cG}{{\mathcal G}}

\newcommand{\cF}{{\mathcal F}}
\newcommand{\cP}{{\mathcal P}}

\newcommand{\cD}{{\mathcal D}}

\newcommand{\cU}{{\mathcal U}}

\newcommand{\cQ}{{\mathcal Q}}
\newcommand{\cZ}{{\mathcal Z}}

\newcommand{\F}{{\mathbb F}}

\newcommand{\GL}{\hbox{{\rm GL}}}

\newcommand{\fq}{{\mathbb F}_{q}}

\newcommand{\PG}{\mathrm{PG}}

\newcommand{\Fix}{\mathrm{Fix}}

\newcommand{\cc}{\mathbf{c}}
\newcommand{\xx}{\mathbf{x}}
\newcommand{\yy}{\mathbf{y}}

\newcommand{\0}{\mathbf{0}}
\newcommand{\PGL}{\mathrm{PGL}}

\newcommand{\dH}{\mathrm{d}_{\mathrm{H}}}
\newcommand{\wH}{\mathrm{w}_{\mathrm{H}}}
\newcommand{\Osc}{\mathrm{Osc}}

\title{On pseudo-arcs from normal rational curve and additive MDS codes}
\author{Francesco Pavese and Paolo Santonastaso}
\date{}

\begin{document}

\maketitle

\begin{abstract}
Let $\PG(k-1,q)$ be the $(k-1)$-dimensional projective space over the finite field $\F_q$. An {\em arc} in $\PG(k-1,q)$ is a set of points with the property that any $k$ of them span the entire space. The notion of a {\em pseudo-arc} generalizes that of an arc by
replacing points with higher-dimensional subspaces.
Constructions of pseudo-arcs can be obtained from arcs defined
over extension fields; such pseudo-arcs are necessarily \emph{Desarguesian}, in the sense
that all their elements belong to a Desarguesian spread.
In contrast, genuinely non-Desarguesian pseudo-arcs are far less understood and have
previously been known only in a few sporadic cases.

In this paper, we introduce a new infinite family of non-Desarguesian pseudo-arcs consisting of $(h-1)$-dimensional projective subspaces of
$\PG(hk-1,q)$ based on the imaginary spaces of a normal rational curve.
We determine the size of the constructed pseudo-arcs explicitly and show that, by adding
suitable osculating spaces of a normal rational curve defined over a subgeometry, we obtain pseudo-arcs of size $O(q^h)$.
As $q$ grows, these sizes asymptotically attain the classical upper bound for pseudo-arcs
established by J.~A.~Thas in \cite{Thas1971}, thereby showing that this bound is essentially sharp also in the
non-Desarguesian setting. We further investigate the interaction between these new pseudo-arcs and quadrics.
While Desarguesian pseudo-arcs from normal rational curve are complete intersections of quadrics, we prove that the
new pseudo-arcs are not contained in any quadric of the
ambient projective space. Finally, we translate our geometric results into coding theory.
We show that the new pseudo-arcs correspond precisely to recent families of additive MDS
codes introduced via a polynomial framework in \cite{neri2025skew}. As a consequence of their non-Desarguesian nature, we prove that these codes are not
equivalent to linear MDS codes.
Moreover, the extension of pseudo-arcs via osculating spaces yields longer families of
additive MDS codes.

\end{abstract}


\section{Introduction}


Let $\F_{q}$ be the finite field of order $q$, where $q$ is a prime power. We denote by $\PG(k-1, q)$ the $(k-1)$-dimensional projective space over $\F_{q}$. 
An {\em arc} in $\PG(k-1,q)$ is a set of points with the property that any $k$ of them span the entire space.  
These objects have a long and rich history, dating back to the foundational works of B. Segre
\cite{segre1955curve,segre1955ovals,segre1959geometrie,segre1967introduction}
and they have played a central role in the development of finite geometry. Also, one of the main reasons for the long-standing interest in arcs is their deep connection with coding theory: arcs can be regarded as the geometric counterparts of linear maximum distance separable (MDS) codes. We refer to
\cite{ball2020arcs} for a comprehensive survey of classical and modern results on arcs.

The notion of arc was later generalized to that of a \emph{pseudo-arc}.
The systematic study of pseudo-arcs goes back to the seminal paper of J. A. Thas \cite{Thas1971}, where these objects were introduced and investigated in terms of arcs in projective spaces over matrix algebras. A \emph{pseudo-arc} in $\PG(hk-1,q)$ is a family of $(h-1)$-dimensional projective spaces
with the property that any $k$ of them span the entire space.
When $h=1$, this definition reduces to the classical notion of an arc in a projective space.
Over the years, pseudo-arcs have attracted growing attention, both for their intrinsic geometric interest and for their connections with other combinatorial and geometric structures. A particularly well-studied case is that of $k=3$.
When a pseudo-arc in $\PG(3h-1,q)$ has size $q^h+1$, it is called a \emph{pseudo-oval}. In this context, pseudo-arcs arise in the study of generalized quadrangles, Laguerre planes and orthogonal arrays, see, for instance,
\cite{penttila2013extending,thas2006translation,VR,ThasSurvey}.

More recently, pseudo-arcs have attracted renewed attention due to their role in coding theory, and in particular in the study of additive codes in the Hamming metric. Additive codes extend the class of linear codes by weakening the linearity assumption on the alphabet. Additive codes in the Hamming metric can also be interpreted as linear codes equipped with the so-called \emph{folded Hamming distance}, see e.g. \cite{martinez2025linear}. From a computer science perspective, codes of this type have been investigated in several contexts, including byte error correction \cite{etzion1998perfect} and the construction of low-density MDS codes \cite{blaum2002lowest,xu1999low}. In recent years, additive codes have also received considerable attention due to their applications in the construction of quantum error-correcting codes; see, for instance,
\cite{dastbasteh2024new,dastbasteh2025polynomial,grassl2021algebraic,ball2023additive,ball2023quantum}.

Additive codes can be studied from a geometric perspective. In \cite{blokhuis2004small}, it was shown that additive codes over $\F_4$ admit a geometric description in terms of multisets of lines in a projective space over $\F_2$. More generally, in \cite{adriaensen2023additive} and \cite{ball2023additive}, the authors proved that additive codes are in one-to-one correspondence with multisets of subspaces in finite projective spaces, thus extending the classical correspondence between linear codes and multisets of points. Further properties of additive codes have been recently investigated from a geometric point of view in \cite{d2025generalized,bartoli2025long,martinez2025linear,krotov2025generalized}.

In this work, we focus on pseudo-arcs in $\PG(hk-1,q)$ and their related additive MDS codes. 


\subsection{Our contribution}

A way to construct pseudo-arcs in $\PG(hk-1,q)$ is by using arcs defined over an extension field. For instance, the normal rational curve in $\PG(k-1,q^h)$ gives rise to a
pseudo-arc of size $q^h+1$ in $\PG(hk-1,q)$.
However, pseudo-arcs obtained in this way have the property that their elements belong to a Desarguesian $(h-1)$-spread of $\PG(hk-1,q)$; we will refer to these types of pseudo-arcs as {\em Desarguesian}. For the definition and properties of a Desarguesian spread, see Section~\ref{sec:Des}.
In this paper, we are primarily interested in pseudo-arcs in $\PG(hk-1, q)$ that are 
\emph{non-Desarguesian}, that is, not contained in any Desarguesian $(h-1)$-spread of $\PG(hk-1, q)$.

Non-Desarguesian pseudo-arcs have so far been studied only in a few small cases, in connection with non-Desarguesian spreads and partial spreads, or
 for $(q,h)\in\{(2,2),(2,3),(3,2),(4,2)\}$  and specific values of $k$ via coding-theoretic
approaches; see \cite{ball2023additive}.
A fundamental result in this area, due to J. A. Thas, provides an upper bound on the size
of a pseudo-arc $\mathcal{A}$ in $\PG(hk-1,q)$:
\begin{equation} \label{eq:boundthas}
    |\mathcal{A}| \le \begin{cases} q^h + k & \mbox{ if } q \mbox{ is even,} \\ q^h + k - 1 & \mbox{ if } q \mbox{ is odd,} \end{cases}
\end{equation}
see \cite[Theorem 4.11]{Thas1971}.
This bound serves as a natural benchmark for the construction of large pseudo-arcs. The first main contribution of this paper is the introduction of a new family of
pseudo-arcs in $\PG(hk-1,q)$ constructed from the imaginary points of a normal rational
curve $\cN_{hk, q^h}$ defined over the extension field $\F_{q^h}$, see \Cref{sec:explicitconstruction}.
Although the construction originates from a normal rational curve, we show that the
resulting pseudo-arc, namely $\mathcal{P}_{h,k,q}$, behaves in a fundamentally different way from the classical
Desarguesian ones.
More precisely, by studying the action of a subgroup $H$ of $\PGL(hk,q)$ acting
$3$-transitively on the normal rational curve in a suitable subgeometry, and by using
recent results connecting Desarguesian spreads and Segre varieties \cite{Sheekeyetal}, we prove that this pseudo-arc is not contained in any Desarguesian $(h-1)$-spread of $\PG(hk-1, q)$.
Consequently, it is not equivalent to a Desarguesian pseudo-arc 
in $\PG(hk-1,q)$, see Section~\ref{sec:nondesarguesian}.

In Section \ref{sec:extending}, we prove that $\mathcal{P}_{h,k,q}$ can be extended by adding the
$(h-1)$-osculating spaces of the normal rational curve $\cN_{hk, q} = \cN_{hk, q^h} \cap \PG(hk-1, q)$. 
This yields a pseudo-arc of size
\[
|\Lambda_{h,q}| + q + 1 = \mathcal{O}(q^h),
\]
where $\Lambda_{h,q}$ denotes a set of representatives of the orbits of size $h$ of $\F_{q^h}$ under the $q$-Frobenius automorphism; as $q$ grows, this size asymptotically attains the upper bound in \eqref{eq:boundthas}. As a consequence, we obtain the largest known family of non-Desarguesian pseudo-arcs
and show that the bound in \eqref{eq:boundthas} is essentially sharp also in the non-Desarguesian setting.

A further contribution concerns the interaction between pseudo-arcs and quadrics, see \Cref{sec:quadrics}.
A normal rational curve is known to be a complete intersection of quadrics, and the related 
Desarguesian pseudo-arc inherits this property. On the other hand, we prove that $\mathcal{P}_{h,k,q}$ is not contained in any quadric of the ambient projective
space.
In particular, it cannot arise as intersections of quadrics.

Finally, in Section \ref{sec:coding}, we translate our geometric constructions into coding-theoretic results.
We prove that the new pseudo-arcs are the geometric counterparts of additive MDS codes
recently constructed via a polynomial framework \cite{neri2025skew}.
Since the pseudo-arcs are non-Desarguesian, it follows that the associated additive MDS
codes are not equivalent to linear MDS, see Section \ref{sec:correspondenceadditive}. 
Moreover, by extending the pseudo-arcs via osculating spaces, we obtain larger families of additive MDS codes, see Section \ref{sec:codeextension}.

\section{Preliminaries}

In this section, we recall the notions and results that will be used throughout the paper. In particular, we review the basic properties of the normal rational curve. Moreover, we recall the notion of Desarguesian spread in a projective space.

\subsection{Normal rational curve and osculating spaces}

An \textbf{arc} in $\PG(k-1,q)$ is a set of points with the property that any $k$ of them span the entire space. 
 One of the most prominent examples of an arc in a projective space is the \emph{normal rational curve}.

\begin{definition}
A \textbf{normal rational curve} $\mathcal{N}_{k,q}$ over $\F_q$ is the set of points in $\PG(k-1,q)$ defined as
\begin{equation} \label{eq:generalrationalcurve}
\bigl\{\,\langle (g_0(u,t),g_1(u,t),\ldots,g_{k-1}(u,t)) \rangle_{\F_q} : u,t \in \F_q, \ (u,t) \neq (0,0) \,\bigr\},
\end{equation}
where $g_0(x,y),g_1(x,y),\ldots,g_{k-1}(x,y) \in \F_q[x,y]$ form a basis for the space of homogeneous polynomials in $\F_q[x,y]$ of degree $k-1$. 
\end{definition}


\noindent A normal rational curve in $\PG(k-1,q)$ consists of $q+1$ points and, when $q \geq k-1$, it forms an arc of $\PG(k-1,q)$, see, e.g., \cite[Theorem 6.30]{HT}. If the condition $q \geq k-1$ is not satisfied, then the normal rational curve contains fewer than $k$ points. Consequently, $\mathcal{N}_{k,q}$ spans at most a hyperplane of $\PG(k-1,q)$ and in this case, $\mathcal{N}_{k,q}$ defines an arc in the subspace spanned by $\mathcal{N}_{k,q}$ itself. For further basic properties of the normal rational curve, we refer the reader to \cite{ball2020arcs,harris2013algebraic,HT}.

In view of \eqref{eq:generalrationalcurve}, the normal rational curve $\mathcal{N}_{k,q}$, can be regarded as the image in
$\PG(k-1,q)$ of the projective line $\PG(1,q)$ under the Veronese
embedding
\begin{equation} \label{eq:polynomialstandardNRC}
    \nu_{k,q}\colon \langle (u,t) \rangle_{\F_q}
    \longmapsto
    \big\langle \big(g_0(u,t),g_1(u,t),\ldots,g_{k-1}(u,t)\big) \big\rangle_{\F_q},
\end{equation}
where \[g_i(x,y)=x^{k-1-i}y^{i} \in \F_q[x,y],\] for every $i\in\{0,\ldots,k-1\}$. We will also denote the image
$
\nu_{k,q}\big(\langle(u,t)\rangle_{\F_q}\big)
$
simply by $\nu_{k,q}(u,t)$ whenever no confusion can arise.
For $j\ge 1$, let
\[
\frac{\partial^{\,j} g_i}{\partial x^{\,j}}
\quad\text{and}\quad
\frac{\partial^{\,j} g_i}{\partial y^{\,j}}
\]
denote the usual $j$--th partial derivatives of the polynomials
$g_i(x,y)$. When $\mathrm{char}(\F_q)=p>j$ and $k-1 > j$, not all coefficients of the partial derivatives are zero.  We can then define the $j$-th derivative of the
parametrization with respect to $x$ and $y$ by
\[
\frac{\partial^{\,j}\nu_{k,q}}{\partial x^{\,j}}(x,y)
:=
\big\langle
\big(\tfrac{\partial^{\,j}}{\partial x^{\,j}} g_0(x,y),
\dots,
\tfrac{\partial^{\,j}}{\partial x^{\,j}} g_{k-1}(x,y)\big)
\big\rangle_{\F_q},
\]
and 
\[
\frac{\partial^{\,j}\nu_{k,q}}{\partial y^{\,j}}(x,y)
:=
\big\langle
\big(\tfrac{\partial^{\,j}}{\partial y^{\,j}} g_0(x,y),
\dots,
\tfrac{\partial^{\,j}}{\partial y^{\,j}} g_{k-1}(x,y)\big)
\big\rangle_{\F_q},
\]
respectively. Each point $\langle (u,t) \rangle_{\F_q} = \langle (1, \frac{t}{u}) \rangle_{\F_q}$, $u \ne 0$, is determined by $\frac{t}{u} \in \F_q$. Therefore, each point of $\PG(1, q)$ is represented by a single element of $\F_q \cup \{\infty\}$, where $\langle (0,1) \rangle_{\F_q}$ corresponds to $\infty$. 
Hence the coordinate functions become $g_i(1,t)=t^i$, and 
the curve $\mathcal{N}_{k,q}$ can also be written as
\begin{equation} \label{eq:standardnormalrationalcurve}
\bigl\{\,\langle (1,t,\ldots,t^{k-1}) \rangle_{\F_q} : t \in \F_q \,\bigr\} \cup \{\,\langle (0,\ldots,0,1)\rangle_{\F_q}\,\}.
\end{equation}
Note that the $j$--th derivative with respect to $t$ is
\[ \frac{\mathrm d^{\,j}}{\mathrm d t^{\,j}} t^i
=
i(i-1)\cdots(i-j+1)\,t^{\,i-j} =\prod_{\ell=0}^{j-1}(i-\ell)\,t^{\,i-j}.
\]
Since $p>j$, no coefficient in this product vanishes in $\F_q$ whenever
$i\ge j$. Hence the $j$--th derivative of the Veronese map
$\nu_{k,q}$ at the point $\langle(1,t)\rangle_{\F_q}$ is
\begin{equation} \label{eq:generatoroscatt}
\frac{\partial^{\,j}}{\partial t^{\,j}}\nu_{k,q}(1,t)
=
\big\langle
\big(
\underbrace{0,\ldots,0}_{j\ \text{zeros}},
\prod_{\ell=0}^{j-1}(j-\ell),
\prod_{\ell=0}^{j-1}(j+1-\ell)\,t,
\prod_{\ell=0}^{j-1}(j+2-\ell)\,t^{2},
\ldots,
\prod_{\ell=0}^{j-1}(k-1-\ell)\,t^{\,k-1-j}
\big)
\big\rangle_{\F_q},
\end{equation}
where the first $j$ entries vanish because the $j$--th derivative of
$t^i$ is zero for $i<j$. Therefore, the formulas above describe the derivatives of the Veronese map at every point of the normal rational curve, except at the point $\langle (0,\ldots,0,1) \rangle_{\F_q}$, that is described below.
The $j$--th derivative of the Veronese map at the point $\langle (u, 1) \rangle_{\F_q}$ 
is
\[
\frac{\partial^{\,j}}{\partial u^{\,j}}\nu_{k,q}(u,1)
=
\langle \bigr(
\prod_{\ell=0}^{j-1}(k-1-\ell)\,u^{\,k-1-j},
\;
\prod_{\ell=0}^{j-1}(k-2-\ell)\,u^{\,k-2-j},
\;
\dots,
\;
\prod_{\ell=0}^{j-1}(j-\ell)\,u^{\,0},
\;
\underbrace{0,\ldots,0}_{j}
\bigl)\rangle_{\F_q}.
\]
Evaluating at $u=0$, the first $k-1-j$ entries vanish. Therefore,
\begin{equation} \label{eq:generatoroscatinfty}
\frac{\partial^{\,j}}{\partial u^{\,j}}\nu_{k,q}(0,1)
=
\langle \bigl(\underbrace{0,\ldots,0}_{k-j-1},1,
\underbrace{0,\ldots,0}_{j}
\bigr) \rangle_{\F_q}.
\end{equation}
Since $p>j$, in both cases each of the sets of points
\[ 
\left\{\nu_{k,q}(1,t),\;
\frac{\partial}{\partial t}\nu_{k,q}(1,t),\;
\dots,\;
\frac{\partial^{\,j}}{\partial t^{\,j}}\nu_{k,q}(1,t) \right\}, \quad \mbox{ with }t \in \F_q,\]
and\[\hskip - 2.4 cm \left\{
\nu_{k,q}(0,1),\;
\frac{\partial}{\partial u}\nu_{k,q}(0,1),\;
\dots,\;
\frac{\partial^{\,j}}{\partial u^{\,j}}\nu_{k,q}(0,1) \right\}
\]
span a $(j-1)$-dimensional projective space of $\PG(hk-1, q)$. So, in view of the discussion above, we are now ready to define the
osculating spaces of a normal rational curve.

\begin{definition}
Assume that the characteristic of the field is $\mathrm{char}(\F_q)=p>j$, for $j\in\{1,\ldots,k-2\}$
and consider the normal rational curve
\[
\mathcal{N}_{k,q}
=\bigl\{\nu_{k,q}(u,t)\;:\;
\langle(u,t)\rangle_{\F_q}\in\PG(1,q)\bigr\} \subseteq \PG(k-1,q).
\]
For every point
$\nu_{k,q}(1,t)$ with $t\in\F_q$, the \textbf{osculating space of order $j$ at}
$\nu_{k,q}(1,t)$ is the $j$-dimensional projective space
\[
\mathrm{Osc}_j\bigl(\nu_{k,q}(1,t)\bigr)
:=
\big\langle
\nu_{k,q}(1,t),
\frac{\partial}{\partial t}\nu_{k,q}(1,t),
\dots,
\frac{\partial^{\,j}}{\partial t^{\,j}}\nu_{k,q}(1,t)
\big\rangle,
\]
whereas the \textbf{osculating space of order $j$ at} $\nu_{k,q}(0,1)$ is the $j$-dimensional projective space
\[
\mathrm{Osc}_j\bigl(\nu_{k,q}(0,1)\bigr)
:=
\big\langle
\nu_{k,q}(0,1),
\frac{\partial}{\partial u}\nu_{k,q}(0,1),
\dots,
\frac{\partial^{\,j}}{\partial u^{\,j}}\nu_{k,q}(0,1)
\big\rangle.
\]
\end{definition}

\noindent When $j=1$, the osculating space of order 1 at a point $P$ is simply the tangent line to the curve at $P$. Also, for every point $P$ one
obtains a natural flag of osculating spaces
\[
P \;\subset\; \mathrm{Osc}_1(P) \;\subset\; \mathrm{Osc}_2(P)
\;\subset\; \cdots \;\subset\; \mathrm{Osc}_{k-2}(P)
.
\]

\subsection{Desarguesian spreads} \label{sec:Des}

An {\em $(h-1)$-spread} of the projective space $\PG(hk-1,q)$ is a family $\mathcal{S}$ of mutually disjoint $(h-1)$-dimensional projective spaces such that every point of $\PG(hk-1,q)$ lies in exactly one element of $\mathcal{S}$. An important class of spreads, called \emph{Desarguesian spreads}, has been described in \cite{segre1964teoria} and can be constructed as follows. Embed $\Sigma \simeq \PG(hk-1,q)$ as a $q$-order subgeometry of $\PG(hk-1,q^h)$ in such a way that $\Sigma$ is the set of fixed points 
\[
\Sigma = \Fix(\Psi)
\] 
of a semilinear collineation $\Psi$ of $\PG(hk-1,q^h)$ of order $h$. For each point $P \in \PG(hk-1,q^h)$, define 
\begin{equation} \label{eq:spaceconjugate}
X^*(P) := \langle P, P^{\Psi}, \ldots, P^{\Psi^{h-1}} \rangle_{q^h}.
\end{equation} Note that for any point $P \in \PG(hk-1,q^h)$, the subspace $X^*(P)$ has dimension at most $h-1$. 
\begin{definition} \label{def:imaginarypoints}
We say that $P$ is an \textbf{imaginary point} of $\PG(hk-1,q^h)$ (with respect to $\Sigma$ and $\Psi$) whenever $X^*(P)$ is an $(h-1)$-dimensional projective space. 
\end{definition}
\noindent Let $\Theta$ be a $(k-1)$-dimensional subspace of $\PG(hk-1,q^h)$ such that 
\begin{equation*} 
    \langle \Theta, \Theta^{\Psi}, \ldots, \Theta^{\Psi^{h-1}} \rangle_{q^h} = \PG(hk-1,q^h).
\end{equation*}
It follows that the space $X^*(P)$ is an $(h-1)$-dimensional projective space of $\PG(hk-1,q^h)$, for every point $P \in \Theta$; in other words, $P$ is an imaginary point of $\PG(hk-1,q^h)$.

\begin{lemma}[see \textnormal{\cite[Lemma 1]{lunardon1999normal}}]\label{lm:subspacesubgeo}
Let $\Pi$ be a subspace of $\PG(hk-1,q^h)$. Then $\langle \Pi \cap \Sigma \rangle_{\F_{q^h}} = \langle \Pi \rangle_{\F_{q^h}}$ if and only if $\Pi^{\Psi} = \Pi$.
\end{lemma}

\noindent By construction, for every $P \in \Theta$, the subspace $X^*(P)$ is fixed by $\Psi$, and hence, by Lemma~\ref{lm:subspacesubgeo}, 
\[
X(P) := X^*(P) \cap \Sigma
\]
is an $(h-1)$-dimensional projective space of $\Sigma$. As $P$ varies in $\Theta$, we obtain the set
\[
\mathcal{D}(\Theta) := \{\, X(P) : P \in \Theta \,\},
\]
which consists of $q^{h(k-1)} + q^{h(k-2)} + \cdots + q^h + 1$ mutually disjoint $(h-1)$-dimensional projective spaces of $\Sigma$, and hence forms an $(h-1)$-spread of $\Sigma$. Any $(h-1)$-spread of the projective space $\PG(hk-1,q)$ that is $\PGL$-equivalent to the construction $\mathcal{D}(\Theta)$ is called a \textbf{Desarguesian spread} of $\PG(hk-1,q)$.  The $(h-1)$-dimensional projective spaces 
$
\Theta, \Theta^{\Psi}, \ldots, \Theta^{\Psi^{h-1}}
$
are uniquely determined by the spread $\mathcal{D}(\Theta)$. More precisely, if $\mathcal{D}(\Theta) = \mathcal{D}(\Theta')$ for some $(k-1)$-dimensional subspace $\Theta'$ of $\PG(hk-1,q^h)$, then necessarily $\Theta' = \Theta^{\Psi^i}$ for some $i \in \{0,\ldots,h-1\}$.  
For this reason, the subspaces $\Theta, \Theta^{\Psi}, \ldots, \Theta^{\Psi^{h-1}}$ are referred to as the \textbf{director spaces} of $\mathcal{D}(\Theta)$.  For further details, see \cite{bader2011desarguesian,lunardon1999normal,segre1964teoria}. 

Let $\Sigma'$ be the subgeometry of $\PG(hk-1, q^h)$ isomorphic to $\PG(hk-1, q)$ consisting of the points represented by
\begin{align*}
    & \left(x_1, x_1^q, \dots, x_1^{q^{h-1}}, x_2,x_2^q, \dots, x_2^{q^{h-1}}, \dots, x_k, x_k^q, \dots, x_k^{q^{h-1}}\right), 
\end{align*}
where $x_i \in \F_{q^h}$, for $i = 1, \dots, k$, with $(x_1, \dots, x_k) \ne (0, \dots, 0)$. Let $\Theta$ be the $(k-1)$-dimensional subspace of $\PG(hk-1,q^h)$ given by 
\begin{align*}
    & X_{2+i} = \ldots = X_{h+i} = 0, & i = 0,h,2h, \dots, h(k-1).
\end{align*}
In this case, $\Sigma'$ is the set of fixed points of the semilinear collineation 
$\Psi'$ of $\PG(hk-1,q^h)$ of order $h$, defined by
\[
\Psi': 
\langle (y_1,\ldots,y_h,y_{h+1},\ldots,y_{2h},\ldots,y_{h(k-1)+1},\ldots,y_{hk}) \rangle_{\F_{q^h}} 
\longmapsto \]
\[
\langle (y_{h}^q,y_1^q,\ldots,y_{h-1}^q,y_{2h}^q,y_{h+1}^q,\ldots,y_{2h-1}^q,\ldots,y_{hk}^q,y_{h(k-1)+1}^q\ldots,y_{hk-1}^q) \rangle_{\F_{q^h}}.
\]

\subsubsection{Automorphism group} \label{sec:autdesarg}

Here we briefly recall the subgroup of $\PGL(hk, q)$ leaving invariant the $(h-1)$-Desarguesian spread $\cD(\Theta)$. The projectivities induced by the matrix 
\begin{align*}
\begin{pmatrix}
    A_{11} & A_{12} & \dots & A_{1k} \\
    A_{21} & A_{22} & \dots & A_{2k} \\
    \vdots & \vdots & \ddots & \vdots \\
    A_{k1} & A_{k2} & \dots & A_{kk}
\end{pmatrix},
\end{align*}
where
\begin{align*}
A_{ij} =
\begin{pmatrix}
    a_{ij} & 0 & \dots & 0 \\
    0 & a_{ij}^q & \dots & 0 \\
    \vdots & \vdots &\ddots & \vdots \\
    0 & 0 & \dots & a_{ij}^{q^{h-1}} 
\end{pmatrix},
\end{align*}
for some $A = (a_{ij}) \in \GL(k, q^h)$, fix $\Sigma'$, $\cD(\Theta)$ and each of its director spaces. On the other hand, the projectivity induced by \begin{align*}
\begin{pmatrix}
    B & \0 & \dots & \0 \\
    \0 & B & \dots & \0 \\
    \vdots & \vdots & \ddots & \vdots \\
    \0 & \0 & \dots & B
\end{pmatrix},
\end{align*}
where
\begin{align*}
B =
\begin{pmatrix}
    \0 & I_{h-1} \\
    1 & \0 
\end{pmatrix}
\end{align*}
generates a cyclic group $C_h$ of order $h$ fixing both $\Sigma'$, $\cD(\Theta)$ and permuting the director spaces of $\cD(\Theta)$ in a single orbit. They give rise to a group of projectivities $\cG$ with structure
\begin{align*}
    \frac{\GL(k, q^h)}{Z(\GL(k, q))} \rtimes C_h.
\end{align*}
Actually, it turns out that $\cG$ contains all the projectivities of $\Sigma'$ leaving $\cD(\Theta)$ invariant, see \cite{Dye, VdV}. In particular, the orbit of a director space of $\cD(\Theta)$ under $\cG$ has size $h$.

\subsubsection{Quadrics}

Let $\F$ be a finite field and let $V$ be an $\F$-vector space.  
A quadratic form over $\F$ is a map $\mathcal{Q}\colon V\to\F$ such that
\begin{itemize}
    \item $\mathcal{Q}(\alpha v)=\alpha^2\mathcal{Q}(v)$, for all $\alpha\in\F,\ v\in V$;
\item 
the map $B(u,v)=\mathcal{Q}(u+v)-\mathcal{Q}(u)-\mathcal{Q}(v)$ is bilinear.

\end{itemize}

The set of points $P=\langle v \rangle_{\F}$ of the projective space $\PG(V,\F)$ that satisfy $\mathcal{Q}(v)=0$ is called \emph{quadric} (or quadratic variety) of $\PG(V, \F)$
associated with $\mathcal{Q}$.  
By abuse of notation, we denote both the quadratic form and the corresponding quadric by $\mathcal{Q}$.

Let $\Sigma'$ be the subgeometry of $\PG(hk-1, q^h)$ isomorphic to $\PG(hk-1, q)$ consisting of the points represented by
\begin{align*}
    & \left(x_1, x_1^q, \dots, x_1^{q^{h-1}}, x_2,x_2^q, \dots, x_2^{q^{h-1}}, \dots, x_k, x_k^q, \dots, x_k^{q^{h-1}}\right), 
\end{align*}
where $x_i \in \F_{q^h}$, for $i = 1, \dots, k$, with $(x_1, \dots, x_k) \ne (0, \dots, 0)$. Let $\Theta$ be the $(k-1)$-dimensional subspace of $\PG(hk-1,q^h)$ given by 
\begin{align*}
    & X_{2+i} = \ldots = X_{h+i} = 0, & i = 0,h,2h, \dots, h(k-1).
\end{align*} 
Starting from a quadric  $\mathcal{Q}$ of $\Theta$, it is possible to construct a family of
quadrics of $\Sigma'$ by means of the so--called \emph{trace trick}.  
More precisely, let $\mathrm{Tr}_{q^h/q}$ denote the trace map from $\F_{q^h}$ to $\F_q$ and, for $\alpha\in\F_{q^h}$, define
\[
L_\alpha\colon \F_{q^h} \to \F_{q}, \qquad x\mapsto \mathrm{Tr}_{q^h/q}(\alpha x).
\]
Then, it is easy to check that the composition
$
L_\alpha\circ\mathcal{Q}\colon V\to\F
$
is a quadratic form over $\F_q$. Therefore, to any quadric $\mathcal{Q}$ of $\Theta$ one can naturally associate a system of quadrics $L_\alpha\circ\mathcal{Q}, \alpha\in\F_{q^h}$ of $\Sigma'$. For further details on this construction, see for instance \cite{gill2007polar,lavrauw2015field}. 

A set of projective subspaces is said to be an \emph{intersection of quadrics} if there exists a family of quadrics whose intersection contains each of its subspaces.  
It is called a \emph{complete intersection of quadrics} if that intersection contains no further points. Here, we note that if a point set $\cZ \subset \Theta \simeq \PG(k-1, q^h)$ is intersection of a family of
quadrics, then the members of
\[X(\cZ) = \{X(P) \;:\; P \in \cZ\}\]
are still contained in the intersection of a suitable
family of quadrics. Moreover, $X(\cZ)$ is a complete intersection of quadrics if $\cZ$ does.

\begin{proposition} \label{prop:pseudoarcquadric}
Let $\cZ \subseteq \Theta \simeq \PG(k-1,q^h)$ be an intersection of quadrics of $\Theta$.  
Then $X(\mathcal{Z})$ is an intersection of quadrics of  $\Sigma' \simeq \PG(hk-1,q)$. Moreover, if $\mathcal{Z}$ is a complete intersection of quadrics, then $X(\mathcal{Z})$ is a complete intersection of quadrics as well.
\end{proposition}
\begin{proof}
Let $\mathcal{Q}_1,\ldots,\mathcal{Q}_{\ell}$ be quadrics of $\Theta \simeq \PG(k-1,q^h)$ such that
$
\mathcal{Z} \subseteq \bigcap_{i=1}^{\ell} \mathcal{Q}_i$. This means that for every point
$
P=\langle v \rangle_{\F_{q^h}} \in \mathcal{Z}$, we have
$
\mathcal{Q}_i(v)=0$, for all  $i=1,\ldots,\ell$. For each $\alpha \in \F_{q^h}^*$, consider the quadratic form $L_\alpha \circ \mathcal{Q}_i$
obtained via the trace $\mathrm{Tr}_{q^h/q}$,  
which defines a quadric of $\Sigma' \simeq \PG(hk-1,q)$.
Since $\mathcal{Q}_i(v)=0$, we also have
\[
(L_\alpha \circ \mathcal{Q}_i)(v) = 0
\qquad \text{for all } \alpha \in \F_{q^h}^*.
\]
Therefore, the $(h-1)$-dimensional projective space $X(P)$ associated with $P$ is contained in
every quadric $L_\alpha \circ \mathcal{Q}_i$, and hence
\begin{equation} \label{eq:containfieldreduction}
X(P) \subseteq \bigcap_{i=1}^{\ell} \bigcap_{\alpha \in \F_{q^h}^*} \bigl( L_\alpha \circ \mathcal{Q}_i \bigr).
\end{equation}
Since this holds for every $P \in \mathcal{Z}$, it follows that $X(\mathcal{Z})$ is an intersection of quadrics of
$\Sigma' \simeq \PG(hk-1,q)$.

Assume now that $\mathcal{Z}$ is a complete intersection of quadrics, that is,
$
\mathcal{Z} = \bigcap_{i=1}^{\ell} \mathcal{Q}_i$.
By the first part of the proof, we have that \eqref{eq:containfieldreduction} holds, for every $P \in \mathcal{Z}$. Let $R=\langle w\rangle_{\F_q}$ be a point of $\Sigma' \simeq \PG(hk-1,q)$ belonging to the right-hand side of \eqref{eq:containfieldreduction}.  
Then $R$ lies on every quadric $L_{\alpha}\circ \mathcal{Q}_i$, for all $\alpha\in\F_{q^h}^*$ and for all $i$. By definition of $L_\alpha$, this yields
\[
\mathrm{Tr}_{{q^h}/q}\!\bigl(\alpha\,\mathcal{Q}_i(w)\bigr)=0
\qquad \text{for all } \alpha\in\F_{q^h}^*.
\]
Since the trace form is non-degenerate, it follows that
$
\mathcal{Q}_i(w)=0,$ and this holds for every $i=1,\ldots,\ell$. Hence the point
$R^*=\langle w \rangle_{\F_{q^h}} \in \Theta \simeq \PG(k-1,q^h)$
belongs to $\bigcap_{i=1}^{\ell} \mathcal{Q}_i=\mathcal{Z}$.  
Consequently, $R \in X(R^*)$, with $R^* \in \mathcal{Z}$, proving that $X(\mathcal{Z})$ is complete intersection of the quadrics $L_{\alpha} \circ \mathcal{Q}_i$, $\alpha \in \F_{q^h}^*$, $i=1,\ldots,\ell$, in $\Sigma' \simeq \PG(hk-1,q)$.
\end{proof}

\section{New constructions of pseudo-arcs} \label{sec:newconstruction}

Recall that a \textbf{pseudo-arc} in $\PG(hk-1,q)$ is a set of $(h-1)$-dimensional projective spaces such that every subset of $k$ of them spans the entire space. In this section, we first present a construction of a pseudo-arc in $\PG(hk-1, q)$, namely $\cP_{h, k, q}$, obtained from the subspaces generated by the imaginary points of a normal rational curve of $\PG(hk-1, q^h)$. 
Next, we prove that $\cP_{h,k, q}$ is not contained in any Desarguesian $(h-1)$-spread of $\PG(hk-1, q)$ and therefore is not equivalent to a Desarguesian pseudo-arc in $\PG(hk-1, q)$. 
Finally, we show how $\cP_{h,k, q}$ can be further extended and investigate its connection with quadrics.


Let $\Sigma \simeq \PG(hk-1,q)$ be a subgeometry of $\PG(hk-1,q^h)$ such that $\Sigma$ is the set of fixed points
$
\Sigma = \Fix(\Psi),
$
where $\Psi$ is a semilinear collineation of $\PG(hk-1,q^h)$ of order $h$. Throughout the remaining part of the paper, we assume that $q \ge hk+1$ and we set $\Sigma$ to be the canonical subgeometry $\PG(hk-1,q)$, i.e., 
\[
\Sigma = 
\Bigl\{
\langle (x_1,\ldots,x_{hk}) \rangle_{\F_{q^h}} 
: (x_1,\ldots,x_{hk}) \in \F_q^{hk},\ (x_1,\ldots,x_{hk}) \neq (0,\ldots,0)
\Bigr\}.
\]
Clearly, $\Sigma$ is the set of fixed points of the semilinear collineation 
$\Psi$ of $\PG(hk-1,q^h)$ of order $h$, defined by
\[
\Psi: 
\langle (x_1,\ldots,x_{hk}) \rangle_{\F_{q^h}} 
\longmapsto 
\langle (x_1^q,\ldots,x_{hk}^q) \rangle_{\F_{q^h}}.
\]

\subsection{Construction} \label{sec:explicitconstruction}

Consider a normal rational curve $\mathcal{N}_{hk,q^h}$ in $\PG(hk-1,q^h)$.  
For an imaginary point $P$ of $\mathcal{N}_{hk,q^h}$, the subspace $X^*(P) = \langle P, P^{\Psi}, \ldots, P^{\Psi^{h-1}} \rangle_{\F_{q^h}}$ is an $(h-1)$-dimensional projective subspace of $\PG(hk-1,q^h)$. We now consider the projective subspaces $X(P)$ obtained by intersecting $X^*(P)$ with the subgeometry $\Sigma$.  
We show that these subspaces form a pseudo-arc in $\PG(hk-1,q)$.

\begin{proposition} \label{prop:constructionofpseudo-arc}
Let $P_1, \ldots, P_{\ell}$ be imaginary points of $\mathcal{N}_{hk,q^h}$ in $\PG(hk-1,q^h)$, no two of which lie in the same orbit under the cyclic group generated by $\Psi$.  
Assume that $\Psi$ stabilizes $\mathcal{N}_{hk,q^h}$ and that $\ell \geq k$.  
Then the set
\[
\{X(P_i), \ i = 1, \ldots, \ell\}
\]
is a pseudo-arc in $\Sigma$.
\end{proposition}

\begin{proof}
Since the points $P_1, \ldots, P_{\ell}$ are imaginary points of $\mathcal{N}_{hk,q^h}$, it follows from the definition and from Lemma \ref{lm:subspacesubgeo} that each $X(P_i)$ is an $(h-1)$-dimensional projective space of $\Sigma$. Suppose, by contradiction, that $k$ of them do not span the whole space $\Sigma$. Then, without loss of generality, we may assume that $X(P_1), \ldots, X(P_k)$ are contained in a hyperplane of $\Sigma$.  
This would imply that the points
\[
P_1, P_1^{\Psi}, \ldots, P_1^{\Psi^{h-1}}, 
P_2, P_2^{\Psi}, \ldots, P_2^{\Psi^{h-1}}, 
\ldots,
P_k, P_k^{\Psi}, \ldots, P_k^{\Psi^{h-1}}
\]
are contained in a hyperplane of $\PG(hk-1,q^h)$.  
However, since $P_1, \ldots, P_k$ belong to distinct orbits under the cyclic group generated by $\Psi$, all the above points are distinct.  
Thus we obtain $hk$ distinct points of the normal rational curve $\mathcal{N}_{hk,q^h}$, that are contained in a hyperplane of $\PG(hk-1,q^h)$, contradicting the fact that $\cN_{hk, q^h}$ is an arc of $\PG(hk-1,q^h)$.  
Therefore, every $k$ of the subspaces $X(P_i)$ span $\Sigma$, and the claim follows.
\end{proof}

From the above result, the question of determining normal rational curves $\mathcal{N}_{hk,q^h}$ stabilized by $\Psi$ naturally arises. We consider it in the following.
  
\begin{proposition}\label{prop:fixed}
Let $\cN_{hk, q^h}$ be a normal rational curve of $\PG(hk-1,q^h)$. Then $\cN_{hk, q^h} \cap \Sigma = \cN_{hk, q}$ if and only if $\cN_{hk,q^h}^{\Psi} = \cN_{hk, q^h}$.
\end{proposition}

\begin{proof}
Recall that $\Sigma=\Fix(\Psi)$, that is,
\begin{equation}\label{eq:SigmaFix}
P\in \Sigma \iff P^\Psi=P.
\end{equation}
Assume first that $\cN_{hk,q^h}\cap\Sigma=\cN_{hk,q}$. Then $\cN_{hk,q}\subseteq \cN_{hk,q^h}$ and
$\cN_{hk,q}$ is pointwise fixed by $\Psi$, by~\eqref{eq:SigmaFix}. Moreover, since
$\lvert \cN_{hk,q} \rvert = q+1 \ge hk+2$, by~\cite[Theorem~6.30]{HT}, $\cN_{hk,q^h}$ is the unique normal rational curve of
$\PG(hk-1,q^h)$ containing $\cN_{hk,q}$. Since $\Psi$ is a collineation of $\PG(hk-1,q^h)$, it follows that
$\cN_{hk,q^h}^{\Psi}$ is also a normal rational curve containing $\cN_{hk,q}$; hence
$\cN_{hk,q^h}^{\Psi}=\cN_{hk,q^h}$. 

Conversely, assume that $\cN_{hk,q^h}^{\Psi}=\cN_{hk,q^h}$.
Let $\nu_{hk,q^h}:\PG(1,q^h)\to \PG(hk-1,q^h)$ be the Veronese embedding and set
\[
\cN'_{hk,q^h}:=\{\nu_{hk,q^h}(u,t):\ \langle(u,t)\rangle_{\F_{q^h}}\in\PG(1,q^h)\}\subseteq \PG(hk-1,q^h).
\]
Every normal rational curve of $\PG(hk-1,q^h)$ is projectively equivalent to $\cN'_{hk,q^h}$. Hence there exists
$A\in \GL(hk,q^h)$ such that
\[
\cN_{hk,q^h}=\{A\, \nu_{hk,q^h}(u,t)^{\top}:\ \langle(u,t)\rangle_{\F_{q^h}}\in\PG(1,q^h)\}.
\]
Since $\cN_{hk,q^h}$ is $\Psi$--invariant, for every $\langle(u,t)\rangle\in\PG(1,q^h)$ there exists a (necessarily unique)
point $\langle(u',t')\rangle\in\PG(1,q^h)$ such that
\begin{equation}\label{eq:inv1}
A^{q}\, \nu_{hk,q^h}(u^q,t^q)^{\top}
=
(A\, \nu_{hk,q^h}(u,t)^{\top})^{\Psi}
=
A\, \nu_{hk,q^h}(u',t')^{\top}.
\end{equation}
Multiplying \eqref{eq:inv1} on the left by $A^{-1}$ yields
\[
S\, \nu_{hk,q^h}(u^q,t^q)^{\top}=\nu_{hk,q^h}(u',t')^{\top},
\qquad\text{where}\qquad
S:=A^{-1}A^{q}\in \GL(hk,q^h).
\]
Thus the projectivity induced by $S$ stabilizes the standard curve $\cN'_{hk,q^h}$. By \cite[Theorem~6.32]{HT}, there exists a matrix $B\in \GL(2,q^h)$ such that the action of $S$ on $\cN'_{hk,q^h}$ is induced
by $B$, that is, for every $\langle(u,t)\rangle_{\F_{q^h}}\in\PG(1,q^h)$,
\begin{equation}\label{eq:SB}
\nu_{hk,q^h}(B\, (u,t)^{\top})
=
S\, \nu_{hk,q^h}(u,t)^{\top}.
\end{equation}

We now determine $\cN_{hk,q^h}\cap \Sigma$. Let $P=A\, \nu_{hk,q^h}(u,t)^{\top}\in \cN_{hk,q^h}$. By~\eqref{eq:SigmaFix},
$P\in\Sigma$ if and only if $P^\Psi=P$, that is,
\begin{equation}\label{eq:fixedcondition}
S\, \nu_{hk,q^h}(u^q,t^q)^{\top}
=
\nu_{hk,q^h}(u,t)^{\top}.
\end{equation}
Note that, by definition of $S$, we have the identity
\[
S\,S^{q}\cdots S^{q^{h-1}}=I.
\]
Equivalently, by~\eqref{eq:SB}, the matrix $B\in \GL(2,q^h)$ (defined up to scalar multiplication) satisfies
\[
B\,B^{q}\cdots B^{q^{h-1}}=I.
\]
By Hilbert’s Theorem~90 (see, for example, \cite[Proposition 3, pag. 151]{serre1979local} and \cite[Proposition 1.3]{Glasby1997writing}), there exists $C\in \GL(2,q^h)$ such that
\begin{equation}\label{eq:coboundary}
B=C^{-1}C^{q}.
\end{equation}
Substituting \eqref{eq:coboundary} into \eqref{eq:SB} and using \eqref{eq:fixedcondition}, we obtain
\[
\nu_{hk,q^h}\!\left(C^{-1}C^{q}(u^q,t^q)^{\top}\right)
=
\nu_{hk,q^h}(u,t),
\]
which implies that $\langle(u,t)\rangle\in \PG(1,q)$.

Therefore, $\cN_{hk,q^h}\cap\Sigma$ is the image under $A \, \nu_{hk,q^h}$ of a $\PG(1,q)$--subline.
Consequently, $\cN_{hk,q^h}\cap\Sigma$ is a normal rational curve in $\Sigma$.
\end{proof}

At this stage, we need to compute the number of imaginary points of a normal rational curve of $\PG(hk-1,q^h)$ fixed by $\Psi$. 
Taking into account Proposition~\ref{prop:fixed},
we can represent the normal rational curve as in \eqref{eq:standardnormalrationalcurve}:
\[
\mathcal{N}_{hk,q^h}
= 
\Bigl\{
\langle (1, t, t^2, \ldots, t^{hk-1}) \rangle_{\F_{q^h}} : t \in \F_{q^h}
\Bigr\}
\cup 
\Bigl\{
\langle (0, \ldots, 0, 1) \rangle_{\F_{q^h}}
\Bigr\}.
\]
We next determine the conditions under which a point of the normal rational curve is imaginary.

\begin{lemma} \label{lem:characimaginarypoints}
The imaginary points of $\mathcal{N}_{hk,q^h}$ are precisely the points of the form 
\begin{equation} \label{eq:formimaginarypoints}
P_{\alpha} = \langle (1, \alpha, \alpha^2, \ldots, \alpha^{hk-1}) \rangle_{\F_{q^h}},
\end{equation}
where $\alpha \in \F_{q^h}$ satisfies $\F_q(\alpha) = \F_{q^h}$.
\end{lemma}

\begin{proof}
For any $\alpha \in \F_{q^h}$ and $i=0,\ldots,h-1$, we have
\[
P_{\alpha}^{\Psi^i} 
= 
\langle (1, \alpha^{q^i}, \alpha^{2q^i}, \ldots, \alpha^{(hk-1)q^i}) \rangle_{\F_{q^h}}.
\]
Consider the matrix whose rows are representatives of the points
$P_{\alpha}, P_{\alpha}^{\Psi}, \ldots, P_{\alpha}^{\Psi^{h-1}}$, namely
\[
M_\alpha=
\begin{pmatrix}
1 & \alpha & \alpha^2 & \cdots & \alpha^{hk-1}\\
1 & \alpha^{q} & \alpha^{2q} & \cdots & \alpha^{(hk-1)q}\\
\vdots & \vdots & \vdots & & \vdots\\
1 & \alpha^{q^{h-1}} & \alpha^{2q^{h-1}} & \cdots & \alpha^{(hk-1)q^{h-1}}
\end{pmatrix} \in \F_{q^h}^{h\times hk}.
\]
This is a (rectangular) $q$-Moore matrix associated with the elements
$1,\alpha,\alpha^2,\ldots,\alpha^{hk-1}$. By \cite[Lemma 3.51]{lidl1997finite}, we get that the rows of $M_\alpha$ are linearly independent over $\F_{q^h}$ if and only if $\alpha$ does not belong to any proper subfield of $\F_{q^h}$,
that is, if and only if $\F_q(\alpha)=\F_{q^h}$. Consequently, the points
\[
P_{\alpha}, P_{\alpha}^{\Psi}, \ldots, P_{\alpha}^{\Psi^{h-1}}
\]
span an $(h-1)$-dimensional projective space if and only if $\alpha$ generates $\F_{q^h}$ over $\F_q$.
This proves the claim.
\end{proof}

In view of \Cref{prop:constructionofpseudo-arc}, we consider imaginary points that do not belong to the same orbit under the cyclic group generated by $\Psi$.  
It is straightforward to verify that this condition corresponds to selecting elements $\alpha \in \F_{q^h}$ lying in distinct orbits of the $q$-Frobenius automorphism. For this reason, we define $\Lambda_{h,q}$ to be a set of representatives of the orbits of size $h$ of $\F_{q^h}$ under the $q$-Frobenius automorphism.  
Note that $\alpha \in \Lambda_{h,q}$ if and only if $\F_q(\alpha) = \F_{q^h}$.  
The size of $\Lambda_{h,q}$ coincides with the number of monic irreducible polynomials over $\F_q$ of degree $h$ and is given by 
\begin{equation*} 
|\Lambda_{h,q}| 
   = \frac{1}{h}\sum_{d \mid h} \mu\!\left(\frac{h}{d}\right) q^d,
\end{equation*}
where $\mu$ denotes the M\"obius function, defined for $m \in \mathbb{N}$ by
    $$\mu(m)=\begin{cases}
        1 & \mbox{ if  } m=1, \\
        (-1)^k & \mbox{ if } m \mbox{  is the product of } k \mbox{ distinct primes, } \\
        0 & \mbox{ if } m \mbox{ is divisible by the square of a prime.}
    \end{cases}$$
See for instance \cite[Section 1.6]{hirschfeld1998projective}. In particular, when $h$ is prime, the expression for $|\Lambda_{h,q}|$ simplifies to
\[
|\Lambda_{h,q}| = \frac{q^h - q}{h}.
\]
The following result holds true.

\begin{theorem} \label{th:explicitconstructionpseudoarc}
The set
\begin{equation} \label{eq:pseudopkhq}
\mathcal{P}_{h,k,q}:=\{X(P_{\alpha}) = X^*(P_{\alpha}) \cap \Sigma, \ \alpha \in \Lambda_{h,q}\},
\end{equation}
where $P_{\alpha}$ is as in \eqref{eq:formimaginarypoints}, is a pseudo-arc 
in $\Sigma \simeq \PG(hk-1,q)$ having size $|\Lambda_{h,q}|$.
\end{theorem}

\begin{remark}
Note that, in order to construct a pseudo-arc as in \Cref{prop:constructionofpseudo-arc}, 
it is enough to consider an arc of $\PG(hk-1,q^h)$ left invariant by $\Psi$.
Hence, the proposition can be extended to any arc of $\PG(hk-1,q^h)$ having this property.  
However, one then needs to determine the imaginary points of such an arc that do not belong
to the same orbit under the cyclic group generated by $\Psi$.
\end{remark}


\subsection{The pseudo-arc \texorpdfstring{$\mathcal{P}_{h,k,q}$}{Phkq} is non-Desarguesian} \label{sec:nondesarguesian}

$
$
Let $\cD$ be a Desarguesian $(h-1)$-spread of $\Sigma$ with director spaces $\Theta, \Theta^{\Psi}, \dots, \Theta^{\Psi^{h-1}}$. 
A natural way to construct a pseudo-arc of $\Sigma$ is the following.  
Consider an arc \( \mathcal{A} \) in \(\Theta \simeq \PG(k - 1, q^h) \). The subset of $\cD$ given by 
\[X(\cA) = \{X(P) \; : \; P \in \cA\}\]
is a pseudo-arc of $\Sigma$.
Since all the elements of a pseudo-arc obtained from an arc $\cA$ in this way belong to the Desarguesian $(h-1)$-spread $\cD$ of $\Sigma$, we will refer to such a pseudo-arc as a \textbf{Desarguesian pseudo-arc}. 
In this section, we prove that the pseudo-arc $\mathcal{P}_{h,k,q}$ is non-Desarguesian, that is, there is no Desarguesian $(h-1)$-spread of $\Sigma$ containing all its members.

Consider the setting as in Section \ref{sec:explicitconstruction}. Let $\cN_{hk, q^h}$ denote the normal rational curve of $\PG(hk-1, q^h)$, $q \ge hk+1$, as in \eqref{eq:standardnormalrationalcurve}. Then $\cN_{hk, q^h} \cap \Sigma = \cN_{hk, q}$, where $\Sigma$ is the canonical subgeometry $\PG(hk-1, q)$. Let $H$ be the subgroup of $\PGL(hk, q)$ of order $q^3-q$ isomorphic to $\PGL(2, q)$ acting $3$-transitively on $\cN_{hk, q} \cap \Sigma$, see \cite[Theorem 6.30]{HT}. 

\begin{lemma}
The group $H$ fixes the pseudo-arc 
\begin{align*}
    \cP_{h,k,q} = \left\{ X(P_{\alpha}) = X^*(P_{\alpha}) \cap \Sigma \;:\; \alpha \in \Lambda_{h,q} \right\}.
\end{align*}
\end{lemma}
\begin{proof}
The group $H$ stabilizes both $\Sigma$ and $\cN_{hk, q^h}$. In particular, if $\xi \in H$, and $P_\alpha$ is a point of $\cN_{hk, q^h}$, where $\alpha$ does not lie in any proper subfield of $\F_{q^h}$, then $P_\alpha^\xi = P_\beta$, with $\beta$ not contained in any proper subfield of $\F_{q^h}$. The assertion follows.
\end{proof}

\begin{lemma}
    There is no Desarguesian $(h-1)$-spread of $\Sigma$ left invariant by $H$.
\end{lemma}
\begin{proof}
Assume, by contradiction, that $\cD(\Theta)$ is a Desarguesian $(h-1)$-spread fixed by $H$. Then $H$ is a subgroup of the group of projectivities $\cG$ of $\Sigma$ leaving invariant $\cD(\Theta)$. See Section \ref{sec:autdesarg}, for a description of $\cG$. Since the orbit of a director space of $\cD(\Theta)$ under $\cG$ has size $h$, it follows that $\Theta^H$, the orbit of $\Theta$ under the action of $H$, has size at most $h$. Hence the stabilizer in $H$ of $\Theta$ is a subgroup of $H$ having order $|H|/|\Theta^H|$, where     
\begin{align*}
    \frac{|H|}{|\Theta^H|} \ge \frac{|H|}{h} > \frac{|H|}{q+1} = q(q-1).
\end{align*}
However it is well known that the largest subgroup of $H$ has size $q(q-1)$, see for instance \cite{Dickson}. Therefore $\Theta$ is fixed by $H$.  
Next, denote by $C$ the subgroup of $H$ of order $q-1$ fixing the points $P_1 = \langle (1,0, \dots, 0) \rangle_{\F_q}$ and $P_2 = \langle (0,0, \dots, 1) \rangle_{\F_q}$ of $\cN_{hk, q}$. Then a member of $C$ is induced by the map
\begin{align*}
    \begin{pmatrix}
        1 & 0 & \dots & 0 \\
        0 & a & \dots & 0 \\
        \vdots & \vdots & \ddots & \vdots \\
        0 & 0 & \dots & a^{hk-1}
    \end{pmatrix},
\end{align*}
for some $a \in \F_q \setminus \{0\}$. Observe that if $T = \langle v \rangle_{\F_{q^h}}$ is a point of $\PG(hk-1, q^h)$, then the points in $T^C$ span a projective subspace of $\PG(hk-1, q^h)$ whose dimension equals the number of non zero entries of $v$ minus one. More precisely, a $(k-1)$-dimensional projective subspace of $\PG(hk-1, q^h)$ that is left invariant by $C$ consists of points with non-zero entries in fixed $k$ positions; hence, it is not disjoint from $\Sigma$. Therefore, $\Theta$ is not disjoint from $\Sigma$. A contradiction.  
\end{proof}

The next result has been obtained in \cite{Sheekeyetal}.

\begin{theorem}\label{thm:intersection}
Let $\cF$ be a pseudo-arc of $\PG(hk-1, q)$ of size $k+1$ and let $\cD$, $\cD'$ be distinct Desarguesian $(h-1)$-spreads of $\PG(hk-1, q)$ such that $\cF \subset \cD \cap \cD'$. Then $|\cD \cap \cD'| = \frac{q^{kr}-1}{q^r-1}$ for some proper divisor $r$ of $h$. Moreover, if $\cD = \cD(\Theta)$, then $\cD \cap \cD'$ corresponds to a subgeometry isomorphic to $\PG(k-1, q^r)$ of $\Theta \simeq \PG(k-1, q^h)$. 
\end{theorem}

Combining the previous result, together with \cite[p. 66]{Hirschfeld1983}, the following holds.

\begin{corollary}\label{cor:maxsize}
Let $\cD$, $\cD'$ be distinct Desarguesian $(h-1)$-spreads of $\PG(hk-1, q)$. The size of a pseudo-arc of $\PG(hk-1, q)$ contained $\cD \cap \cD'$ cannot exceed $q^r+k$, for some proper divisor $r$ of $h$.  
\end{corollary}

We are now ready to prove that the pseudo-arc $\mathcal{P}_{h,k,q}$ is non-Desarguesian.

\begin{theorem} \label{th:ourarcisnondesarguesian}
    The pseudo-arc 
    \begin{align*}
    \cP_{h,k,q} = \left\{ X(P_{\alpha}) = X^*(P_{\alpha}) \cap \Sigma \;:\; \alpha \in \Lambda_{h,q} \right\}
\end{align*}
is non-Desarguesian.
\end{theorem}
\begin{proof}
Assume, by contradiction, that $\cD$ is a Desarguesian $(h-1)$-spread of $\Sigma$ such that $\cP_{h,k,q} \subset \cD$. Since $H$ fixes $\cP_{h,k,q}$, but there is no Desarguesian $(h-1)$-spread left invariant by $H$, it follows that there is a Desarguesian $(h-1)$-spread of $\Sigma$, say $\cD'$, contained in $\cD^H$, with $\cD \ne \cD'$, and such that $\cP_{h,k,q} \subset \cD \cap \cD'$. Therefore, 
\begin{align*}
    |\cP_{h,k,q}| \le q^r + k < q^{h-1},
\end{align*}
by Corollary \ref{cor:maxsize}. On the other hand, 
\begin{align*}
|\cP_{h,k,q}| = |\Lambda_{h, q}| \ge \frac{q^h - q^{h-1}}{h} > (k-1)q^{h-1} \ge q^{h-1}, \end{align*}
a contradiction.
\end{proof}

\subsection{Extending the pseudo-arc \texorpdfstring{$\mathcal{P}_{h,k,q}$}{Phkq}} \label{sec:extending}

A pseudo-arc is said to be {\em complete} if it is not contained in a larger pseudo-arc.  
The study of the completeness of the normal rational curve has been addressed in a series of papers \cite{seroussi2003mds,storme1991generalized,storme1992completeness}.  
In this section, we show that it is possible to enlarge the pseudo-arc $\cP_{h,k,q}$ in $\Sigma$, which  arises from the normal rational curve $\cN_{hk, q^h}$ of $\PG(hk-1,q^h)$, by adding the osculating spaces of the normal rational curve $\cN_{hk, q}$ of $\Sigma$. 
 Throughout, we keep the notation of
\Cref{sec:explicitconstruction} and assume that $\mathrm{char}(\F_q)=p\geq h$. Again, let $\cN_{hk, q^h}$ denote the normal rational curve of $\PG(hk-1, q^h)$, $q \ge hk-1$, as in \eqref{eq:standardnormalrationalcurve}. Then $\cN_{hk, q^h} \cap \Sigma = \cN_{hk, q}$, where $\Sigma$ is the canonical subgeometry $\PG(hk-1, q)$.
For each point $Q\in\mathcal{N}_{hk,q}$, let $\Osc_{h-1}(Q)$ be the osculating
space of order $h-1$ at $Q$, which is an $(h-1)$-dimensional projective space of $\Sigma$ since $p\geq h$.
Consider the family
\[
\mathcal{O}_{h-1}
\;:=\;
\Bigl\{\,
\Osc_{h-1}\bigl(\nu_{hk,q}(1,t)\bigr)\ :\ t\in\F_q
\,\Bigr\}
\ \cup\
\Bigl\{\,\Osc_{h-1}\bigl(\nu_{hk,q}(0,1)\bigr)\,\Bigr\},
\]
which consists of $q+1$ projective subspaces of dimension $h-1$ in $\Sigma$.

\begin{theorem}\label{th:extension_pseudoarc}
Assume that $\mathrm{char}(\F_q)=p\geq h$. Then \[\cP_{h,k,q} \cup\ \mathcal{O}_{h-1}
\]
is a pseudo-arc 
in
$\Sigma \simeq \PG(hk-1,q)$ of size $|\Lambda_{h,q}|+q+1$.
\end{theorem}

\begin{proof} 
By \Cref{th:explicitconstructionpseudoarc}, we know that $\cP_{h,k,q}$ is a pseudo-arc.
We need to show that adding the osculating spaces in $\mathcal{O}_{h-1}$ preserves
the pseudo-arc property. For each $t\in\F_q$, let $A_t$ be the matrix whose rows form a basis for the osculating space of order $h-1$ of $\mathcal{N}_{hk,q}$ at the point
$\nu_{hk,q}(1,t)$.  
According to \eqref{eq:generatoroscatt}, we can write
\begin{equation} \label{eq:basisosculating}
A_t :=
\begin{pmatrix}
1 & t & t^{2}  & \cdots & t^{hk-1} \\[4pt]
0 & 1 & 2t & \cdots & (hk-1)t^{hk-2} \\[4pt]
0 & 0 & 2\!\cdot\!1 & \cdots & (hk-1)(hk-2)t^{hk-3} \\[4pt]
\vdots &  \vdots & \vdots & \ddots & \vdots \\[4pt]
0 & \cdots & 0 &
\cdots &
\prod\limits_{i=1}^{h-1}(hk-i)\,t^{hk-h}
\end{pmatrix}
\in \F_q^{h\times hk}.
\end{equation}
According to \eqref{eq:generatoroscatinfty}, the osculating space of order $h-1$ at
$\nu_{hk,q}(0,1)$ is spanned by the rows of the matrix
\begin{equation}  \label{eq:basisosculatinginfty}
A_{\infty}=
\begin{pmatrix}
0 & \cdots & 0 & 0 & \cdots & 0 & 1 \\[4pt]
0 & \cdots & 0 & 0 & \cdots & 1 & 0 \\[4pt]
0 & \cdots & 0 & 0 & 1 & 0 & 0 \\[2pt]
\vdots &  & \vdots & \reflectbox{$\ddots$} & \vdots & \vdots & \vdots \\[2pt]
0 & \cdots & 1 & 0 & 0 & 0 & 0
\end{pmatrix} \in\F_q^{\,h\times hk}.
\end{equation}

\noindent Finally, for each $\alpha\in\Lambda_{h,q}$, define $B_{\alpha}$ as the matrix whose
rows are the vectors representing the points
$P_{\alpha},P_{\alpha}^{\Psi},\ldots,P_{\alpha}^{\Psi^{h-1}}$, namely
\[
B_{\alpha} =
\begin{pmatrix}
1 & \alpha & \alpha^{2} & \cdots & \alpha^{hk-1} \\[4pt]
1 & \alpha^{q} & \alpha^{2q}  & \cdots & \alpha^{(hk-1)q} \\[4pt]
1 & \alpha^{q^{2}} & \alpha^{2q^{2}}  & \cdots & \alpha^{(hk-1)q^{2}} \\[4pt]
\vdots & \vdots & \vdots & \ddots & \vdots \\[4pt]
1 & \alpha^{q^{h-1}} & \alpha^{2q^{h-1}}  & \cdots & \alpha^{(hk-1)q^{\,h-1}}
\end{pmatrix}
\in \F_{q^h}^{h\times hk}.
\]
The row space of $B_{\alpha}$ spans the subspace $X^*(P_{\alpha})$ associated with
the orbit of $P_{\alpha}$ under $\Psi$. To prove the assertion, it is enough to show that for any
$j\in\{1,\ldots,k\}$, if we take $j$ pairwise distinct elements
$t_1,\ldots,t_{j}\in\F_q \cup \{\infty\}$ and
$k-j$ pairwise distinct elements
$\alpha_1,\ldots,\alpha_{k-j}\in \Lambda_{h,q}$,
then the block matrix
\[
N :=
\begin{pmatrix}
A_{t_1}\\[2pt]
\vdots\\[2pt]
A_{t_{j}}\\[2pt]
B_{\alpha_1}\\[2pt]
\vdots\\[2pt]
B_{\alpha_{k-j}}
\end{pmatrix}
\in \F_{q^h}^{hk\times hk}
\]
is nonsingular. When $j=k$, only the matrices $A_{t_i}$ appear in $N$.

\noindent Assume, for contradiction, that $N$ is singular.  
Then there exists a nonzero vector
\[
\mathbf{c}=(c_0,\ldots,c_{hk-1})\in\F_{q^h}^{hk}
\]
such that $N\mathbf{c}^{\top}=\mathbf{0}$.
This means that
\[
A_{t_i}\mathbf{c}^{\top}=\mathbf{0}\quad
\text{for all }i\in\{1,\ldots,j\},
\qquad
B_{\alpha_r}\mathbf{c}^{\top}=\mathbf{0}\quad
\text{for all }r\in\{1,\ldots,k-j\}.
\]

\smallskip
\noindent
Define the polynomial
\[
f(Z)=\sum_{i=0}^{hk-1}c_i Z^i\in\F_{q^h}[Z].
\]
From $B_{\alpha_r}\mathbf{c}^{\top}=\mathbf{0}$ we deduce that
\[
f(\alpha_r^{q^s})=0,
\qquad
r\in\{1,\dots,k-j\},\ s\in\{0,1,\dots,h-1\},
\]
since the $s$-th row of $B_{\alpha_r}$ corresponds to the evaluation of
$f$ at $\alpha_r^{q^s}$.
Because $\F_q(\alpha_r)=\F_{q^h}$, the conjugates
$\alpha_r,\alpha_r^{q},\ldots,\alpha_r^{q^{h-1}}$
are $h$ distinct elements of $\F_{q^h}$.
Hence each $\alpha_r$ contributes $h$ distinct simple zeros of $f$.

\smallskip
\noindent
From $A_{t_i}\mathbf{c}^{\top}=\mathbf{0}$, if $t_i\neq\infty$ we obtain
\[
f^{(\ell)}(t_i)=0,
\qquad
\text{for all }\ell\in\{0,1,\dots,h-1\},
\]
where $f^{(\ell)}$ denotes the $\ell$-th derivative of $f$.
Thus each $t_i$ is a root of multiplicity at least $h$. In total, $f$ has at least $j\cdot h$ zeros (counted with multiplicity)
coming from the $t_i$, and $(k-j)\cdot h$ distinct simple zeros
coming from the $\alpha_r$.  
Hence $f$ has at least $hk$ zeros in $\F_{q^h}$ counting multiplicities.
Since $\deg f\le hk-1$, this is impossible for a nonzero polynomial.  
Therefore $f$ must be the zero polynomial, implying $\mathbf{c}=\mathbf{0}$.
Hence $N$ is nonsingular. Finally, if one of the parameters $t_i=\infty$, then
the last $h$ coordinates of $\mathbf{c}$ vanish,
that is, $c_{hk-h}=c_{hk-h+1}=\cdots=c_{hk-1}=0$.
In this case $f$ has degree at most $hk-h-1$, and
the same argument as above applies verbatim,
yielding the desired contradiction.
\end{proof}


\begin{remark}
The question whether the pseudo-arc $\cP_{h,k,q} \cup \mathcal{O}_{h-1}$ is complete naturally arises. We remark that, in the case $(h,k)=(2,2)$, the pseudo-arc
$\mathcal{P}_{2,2,q} \cup \mathcal{O}_1$ in $\PG(3, q)$ 
coincides with the one presented
in \cite[Section 21.4]{hirschfeld1985finite}. In this case the pseudo-arc is complete if $q \equiv 1 \pmod{3}$  \cite[Theorem 21.4.1]{hirschfeld1985finite}; on the other hand, if $q \equiv -1 \pmod{3}$, $q > 2$, by adding the imaginary axis, a non-Desarguesian spread of lines of $\PG(3, q)$ is obtained, see \cite[Theorem 21.4.2]{hirschfeld1985finite}.
\end{remark}

\subsection{Pseudo-arcs and quadrics} \label{sec:quadrics}

Quadrics are intimately connected with arcs. For instance, in \cite{glynn1994construction} it is proved that every arc in $\PG(k-1,q)$ having size $2k-1$ is contained in the intersection of a unique
projective space $S$ of quadrics of dimension $\binom{k-1}2-1$ that
does not contain a quadric that is the union of two hyperplanes. Moreover, a normal rational curve in
$\PG(k-1,q)$ is the complete intersection of such a space $S$ of quadrics. We refer to \cite{glynn1994construction} for further results
of this type. In the following, we investigate this relationship in the context of
pseudo-arcs. A pseudo-arc is said to be an \emph{intersection of quadrics} if there exists a family of quadrics whose intersection contains each of its subspaces.  
It is called a \emph{complete intersection of quadrics} if that intersection contains no further points besides those lying in the union of its subspaces. By Proposition~\ref{prop:pseudoarcquadric}, we see that if an arc $\cA$ in $\Theta \simeq \PG(k-1, q^h)$ is intersection of a family of
quadrics, then the members of the associated pseudo-arc
\[X(\cA) = \{X(P) \;:\; P \in \cA\}\]
are still contained in the intersection of a suitable
family of quadrics. Moreover, $X(\cA)$ is a complete intersection of quadrics if $\cA$ does.

\begin{proposition} \label{prop:pseudoarcquadric1}
Let $\mathcal{A}\subseteq \Theta \simeq \PG(k-1,q^h)$ be an arc, and assume that it is an intersection of quadrics in $\PG(k-1,q^h)$.  
Then the associated Desarguesian pseudo-arc $X(\mathcal{A})$ of $\PG(hk-1,q)$ is an intersection of quadrics in $\PG(hk-1,q)$.  
Moreover, if $\mathcal{A}$ is a complete intersection of quadrics, then the associated pseudo-arc $X(\mathcal{A})$ is a complete intersection of quadrics as well.
\end{proposition}

By \cite[Theorem 3.3]{glynn1994construction}, a normal rational curve of $\PG(k-1,q)$ is the complete intersection of a space of quadrics not containing any hyperplane-pair. More precisely, consider the normal rational curve 
\[
\mathcal{N}_{k,q}=\bigl\{\,\langle (1,t,t^2,\ldots,t^{k-1})\rangle_{\fq} : t\in\F_q\,\bigr\}\cup\{\langle(0,\ldots,0,1)\rangle_{\fq}\}.
\]
The space of quadrics whose complete intersection is $\mathcal{N}_{k,q}$ is generated by the quadrics
\begin{equation} \label{eq:quadricsrationalnormalcurve}
\cQ_{i,j}\colon\; x_i x_j-x_{i+1}x_{j-1}=0,
\qquad 1\le i \leq j-2, \ 2 \leq j\le k,
\end{equation}
of $\PG(k-1,q)$ which span a vector space of dimension $\binom{k-1}{2}$.  
None of these quadrics is a union of two hyperplanes, and their intersection coincides exactly with the normal rational curve. Clearly, we can apply this description to a normal rational curve $\mathcal{N}_{k,q^h}$ of $\Theta \simeq \PG(k-1,q^h)$. By Proposition \ref{prop:pseudoarcquadric1}, the Desarguesian pseudo-arc $X(\mathcal{N}_{k,q^h})$ of $\PG(hk-1,q)$, associated with $\mathcal{N}_{k,q^h}$, is a complete intersection of quadrics as well.
 
Now, we investigate the situation for the pseudo-arc $\cP_{h,k,q}$ constructed from the imaginary points of a normal rational curve as in \Cref{th:explicitconstructionpseudoarc}. We prove that, despite the fact that it also originates from points of a normal rational curve, its geometric behaviour is fundamentally different: no quadric of the ambient projective space contains all of its elements. Consequently, this pseudo-arc is not an intersection of quadrics.

We consider the same setting as in Section \ref{sec:explicitconstruction}, with $\Sigma\cong\PG(hk-1,q)$, $\Psi$ and the normal rational curve
$
\mathcal{N}_{hk,q^h}$ as defined therein.

\begin{theorem}
There is no quadric of $\PG(hk-1,q)$ containing all the
$(h-1)$-dimensional projective spaces of the pseudo-arc $\cP_{h,k,q}$. In particular, $\cP_{h,k,q}$ is not an
intersection of quadrics.
\end{theorem}
\begin{proof}
Assume, by contradiction, that there exists a quadric $\mathcal{Q}$ of $\PG(hk-1,q)$
containing all the $(h-1)$-dimensional projective spaces
\[
X(P_{\alpha}) = X^*(P_{\alpha})\cap \Sigma,\qquad \alpha\in\Lambda_{h,q},
\]
where $P_{\alpha}$ is as in \eqref{eq:formimaginarypoints}.
Let $\mathcal{Q}^*$ denote the extension of $\mathcal{Q}$ to
$\PG(hk-1,q^h)$, obtained by viewing the defining homogeneous quadratic equation
of $\mathcal{Q}$ over $\F_{q^h}$.
Since $\mathcal{Q}$ contains $X(P_{\alpha})$, it follows that $\mathcal{Q}^*$ contains
$X^*(P_{\alpha})$.
In particular, $\mathcal{Q}^*$ contains the points
$P_{\alpha},P_{\alpha}^{\Psi},\ldots,P_{\alpha}^{\Psi^{h-1}}$ for every
$\alpha\in\Lambda_{h,q}$.
These are pairwise distinct points of the normal rational curve
$\mathcal{N}_{hk,q^h}$.
Therefore, $\mathcal{Q}^*$ contains the sub-arc
\[
\mathcal{A}_0:=\{\,P_{\alpha}^{\Psi^i} : \alpha\in\Lambda_{h,q},\ i=0,\ldots,h-1\,\}
\subseteq \mathcal{N}_{hk,q^h},
\]
whose size is
\[
|\mathcal{A}_0|
= h|\Lambda_{h,q}|
\ge q^h-\sum_{\substack{d\mid h\\ d<h}} q^d
\ge q^h-q^{h-1}.
\]
In particular, since $q\ge hk$, then $|\mathcal{A}_0|\ge 2hk-1$.
Hence, by \cite[Theorem~3.3]{glynn1994construction}, the quadric $\mathcal{Q}^*$ vanishes
on the whole normal rational curve $\mathcal{N}_{hk,q^h}$ and must belong to the space of
quadrics defining $\mathcal{N}_{hk,q^h}$, namely the space generated by the quadrics of
$\PG(hk-1,q^h)$
\[
\mathcal{Q}_{i,j}^*:\ x_i x_j-x_{i+1}x_{j-1}=0,
\qquad 1 \le i \le j-2,\ 3\le j\le hk,
\]
cf.~\eqref{eq:quadricsrationalnormalcurve}.
In particular, we may write
\[
\mathcal{Q}^*:
\sum_{\substack{1 \le i \le j-2\\ 3 \le j \le hk}}
\lambda_{i,j}\,(x_i x_j-x_{i+1}x_{j-1})=0,
\]
for some coefficients $\lambda_{i,j}\in\F_{q^h}$, not all zero. It can then be checked by direct computation that the bilinear form
\[
B^*(\mathbf{x},\mathbf{y})=\sum_{i,j}\gamma_{i,j}x_i y_j,
\]
with $\mathbf{x}=(x_1,\ldots,x_{hk})$ and $\mathbf{y}=(y_1,\ldots,y_{hk})$, is the bilinear
form associated with $\mathcal{Q}^*$, where the coefficients $\gamma_{i,j}$ are given by
\[
\begin{cases}
\gamma_{1,1}=\gamma_{1,2}=\gamma_{hk-1,hk}=\gamma_{hk,hk}=0,\\
\gamma_{1,j}=\lambda_{1,j}, & \text{for all } j\ge 3,\\
\gamma_{i,hk}=\lambda_{i,hk}, & \text{for all } i\le hk-2,\\
\gamma_{i,i}=-2\lambda_{i-1,i+1}, & \text{for all } 2\le i\le hk-1,\\
\gamma_{i,i+1}=-\lambda_{i-1,i+2}, & \text{for all } 2\le i\le hk-2,\\
\gamma_{j,i}=\gamma_{i,j}=\lambda_{i,j}-\lambda_{i-1,j+1},
& \text{for all } 2\le i\le j-2,\ 4\le j\le hk-1.
\end{cases}
\]

Since the coefficients $\lambda_{i,j}$ are not all zero, it follows from the above
expressions that the bilinear form $B^*$ is not identically zero.
Moreover, since $X^*(P_\alpha)\subseteq \mathcal{Q}^*$ for every $\alpha\in\Lambda_{h,q}$,
and since $P_\alpha,P_\alpha^{\Psi},\ldots,P_\alpha^{\Psi^{h-1}}$ belong to
$X^*(P_\alpha)$, it follows from \cite[Lemma~1.5]{HT} that
\begin{equation} \label{eq:condbilinear}
B^*(P_\alpha^{\Psi^i},P_\alpha^{\Psi^{i+1}})=0
\end{equation}
for every $\alpha\in\Lambda_{h,q}$ and $i=0,\ldots,h-1$.

Now, let $\mathbf{Z}=(1,Z,\ldots,Z^{hk-1})$, and consider the polynomial
\[
f(Z):=B^*(\mathbf{Z},\mathbf{Z}^q)
=\sum_{i,j}\gamma_{i,j}Z^{i-1}Z^{q(j-1)}.
\]
Since $q\ge hk$ and $B^*$ is non-zero, $f(Z)$ is a non-zero polynomial over $\F_{q^h}$ of
degree at most $(hk-1)(q+1)$.
By~\eqref{eq:condbilinear}, for every $\alpha\in\Lambda_{h,q}$ the elements
$\alpha,\alpha^q,\ldots,\alpha^{q^{h-1}}$ are zeros of $f(Z)$.
Hence, $f(Z)$ has at least
\[
h|\Lambda_{h,q}|
\ge q^h-q^{h-1}
\]
distinct zeros.
In particular, for $q\ge hk+1$ this number is strictly larger than $\deg(f)\le (hk-1)(q+1)$,
which contradicts the degree bound.
Therefore, no such quadric $\mathcal{Q}$ exists.
\end{proof}

\section{Applications to additive MDS codes} \label{sec:coding}

Pseudo-arcs have recently regained considerable attention, especially because of their connection with additive MDS codes.  
Indeed, in \cite{ball2023additive} it was shown that they represent the geometric counterpart of additive MDS codes.  
In this section, we first recall the necessary background on additive codes and their geometric interpretation.  
We then prove that the codes associated with the pseudo-arc $\mathcal{P}_{h,k,q}$ constructed in Theorem \ref{th:explicitconstructionpseudoarc} correspond precisely to the additive MDS codes arising from a polynomial construction recently introduced in \cite{neri2025skew}.  
As a consequence, we show that these codes are not equivalent to any linear MDS code.  
Finally, by considering the extension of the pseudo-arc $\mathcal{P}_{h,k,q}$ via osculating spaces, we can also extend the codes constructed in \cite{neri2025skew}.

\subsection{Basics on additive Hamming metric codes} We begin by recalling some basic notions of Hamming-metric codes that will be needed in the following. For further background, we refer the reader to \cite{ball2020course}.

Let $n$ be a positive integer and let $\F$ denote a finite field.  
The \emph{Hamming distance} on $\F^n$ is defined by
\[
\dH(\mathbf{x},\mathbf{y}) := \lvert \{\, i : x_i \neq y_i \,\} \rvert,
\]
for all vectors $\xx=(x_1,\ldots,x_n)$ and $\yy=(y_1,\ldots,y_n)$ in $\F^n$.  
A \emph{Hamming-metric code} is simply a subset of the metric space $(\F^n,\dH)$. A code $\C$ is called \emph{additive} if it is closed under vector addition, i.e., if $\xx,\yy \in \C$ implies $\xx+\yy \in \C$.  
Suppose now that the alphabet is $\F = \F_{q^h}$, the finite field with $q^h$ elements.  
Viewing $\F_{q^h}$ as a vector space over its subfield $\F_q$, we say that a code $\C \subseteq \F_{q^h}^n$ is \emph{$\F_q$-linear} if it is an $\F_q$-subspace of $(\F_{q^h}^n, \dH)$.  

We write that $\C$ is an \emph{$\F_q$-linear $(n,q^k,d)_{q^h}$ code} if $\C$ is an $\F_q$-linear Hamming-metric code in $\F_{q^h}^n$, $\dim_{\F_q}(\C) = k$, and $d$ is the \emph{minimum distance}:
    \[
    d = \dH(\C) := \min\{ \dH(\mathbf{c}_1, \mathbf{c}_2) : \mathbf{c}_1, \mathbf{c}_2 \in \C,\ \mathbf{c}_1 \ne \mathbf{c}_2 \}.
    \]
If the minimum distance is not known or not relevant, we simply denote $\C$ as an $\F_q$-linear $(n,q^k)_{q^h}$ code.  
Note that every additive code is $\F_p$-linear, and every $\F_q$-linear code is additive. When $\C$ is $\F_{q^h}$-linear, we refer to it simply as a \emph{linear} code. The \emph{Hamming weight} of a vector $\xx = (x_1,\ldots,x_n) \in \F_{q^h}^n$ is defined as
\[
\wH(\xx) := \dH(\xx, \mathbf{0}) = \lvert \{\, i : x_i \neq 0 \,\} \rvert.
\] 
Since $\dH(\xx,\yy) = \wH(\xx - \yy)$ for all $\xx, \yy \in \F_{q^h}^n$, the minimum distance of an additive code $\C$ can be equivalently expressed as
$\dH(\C) = \min\{\, \wH(\cc) : \cc \in \C \setminus \{\mathbf{0}\} \,\},$
see e.g. \cite[Lemma 3]{ball2023additive}. The fundamental trade-off between the parameters of a Hamming-metric code is expressed by the classical Singleton bound: if $\C$ is an $\F_q$-linear $(n,q^k,d)_{q^h}$ code, then
    \[
        q^k \;\leq\; q^{h(n-d+1)}.
    \]
\noindent Codes that attain equality in the Singleton bound are called \emph{maximum distance separable} codes, or simply \emph{MDS codes}.

\begin{definition} \label{def:additiveequivalence}
   Let $\C$ and $\C'$ be $\F_q$-linear Hamming-metric codes in $\F_{q^h}^n$.  
   We say that $\C$ and $\C'$ are \textbf{$\F_q$-linearly equivalent} if there exist $\F_q$-linear automorphism $\phi_1,\ldots,\phi_n$ of $\F_{q^h}$ and a permutation $\theta$ of $\{1,\ldots,n\}$ such that
   \[
   \C' = \left\{ \big(\phi_1(x_{\theta(1)}),\ldots,\phi_n(x_{\theta(n)})\big) : (x_1,\ldots,x_n) \in \C \right\}.
   \]
\end{definition}

Codes that are $\F_q$-linear equivalent clearly have the same size and the same minimum distance.  
Note that this notion of additive equivalence includes the \emph{monomial equivalence}, where each $\phi_i$ is of the form 
    $
    y \mapsto \alpha_i y, $ with $\alpha_i \in \F_{q^h}^*,
    $
    in which case the codes are also called \emph{monomial (or linear) equivalent}, see e.g. \cite{huffman2010fundamentals}.

\begin{remark}
When $q=p$ is prime, the notion of $\F_p$-linear equivalence coincides with the general notion of equivalence for Hamming-metric codes.  
Indeed, in full generality two Hamming-metric codes $\C$ and $\C'$ are said to be \textbf{equivalent} if the maps $\phi_i$ in \Cref{def:additiveequivalence} are arbitrary permutations of $\F_{p^h}$.  
However, as shown in \cite{ball2022equivalence}, for additive MDS codes the appropriate notion of equivalence is precisely additive equivalence, i.e., $\F_p$-linear equivalence.  Indeed, \cite[Theorem 4]{ball2022equivalence} shows that if $\C$ and $\C'$ are additive MDS codes in $\F_{p^h}^n$, then $\C$ and $\C'$ are equivalent if and only if they are $\F_p$-linearly equivalent. It should also be noted that, as observed in \cite[Section 4]{ball2022equivalence}, it is not known whether equivalence can in general be characterized in terms of additive equivalence for all additive codes. Our study, however, is restricted to additive MDS codes, for which, in the case where $q=p$, this characterization applies. For further details on equivalence in Hamming-metric spaces, we refer the reader to \cite{ball2022equivalence,betten2006error} and the references therein.
\end{remark}

\medskip
Finally, we recall that in the sequel we only consider Hamming-metric codes that are \textbf{nondegenerate}, meaning that the code cannot be isometrically embedded into a smaller ambient space.  

\subsection{The geometry of additive Hamming metric codes} 

There is a well-known geometric interpretation of codes endowed with the Hamming metric.  
A linear $(n,q^k,d)_q$ code $\C$ corresponds to a multiset of points in $\PG(k-1,q)$, and conversely, see e.g. \cite{MR1186841}.  
In \cite{ball2023additive}, this correspondence was extended to additive codes over finite fields, see also \cite{adriaensen2023additive}.  
Below we briefly recall this viewpoint.







A \textbf{generator matrix} of a $(n, q^{hk}, d)_{q^h}$ code $\cC$, which is linear over $\F_q$, is an $hk \times n$ matrix
with entries from $\F_{qh}$, whose row space over $\F_q$ is $\cC$. Fix an $\F_q$-basis $\{\omega_1,\ldots,\omega_h\}$ of $\F_{q^h}$ and let $\boldsymbol{\omega}= (\omega_1,\ldots,\omega_h)^{\top}$. Then the $j$-th column of $G$ is of the form $G_j \boldsymbol{\omega}$, for some $G_j \in \F_q^{hk \times h}$.
Define $\pi_j$ as the subspace of $\PG(hk-1,q)$ corresponding to the column space of $G_j$, for each $j \in \{1,\ldots, n\}$.  
Each $\pi_j$ has projective dimension at most $h-1$. Also the subspace $\pi_j$ is independent of the chosen $\F_q$-basis of $\F_{q^h}$, see \cite[Remark 3.1]{adriaensen2023additive}. Therefore, we can consider the multiset $\mathcal{U}=\{\pi_1,\ldots,\pi_n\}$.
Conversely, given a multi-set of subspaces of $\PG(hk-1, q)$ of projective dimension at most $h-1$,
once that an $\F_q$-basis $\{\omega_1,\ldots,\omega_h\}$ of $\F_{q^h}$ is fixed, an additive code $\cC$ can be constructed by reversing the
above process.

\begin{theorem}
$\cC$ is an additive $(n, q^{kh}, d)_{q^h}$ code which is linear over $\F_q$ if and only if the multi-set $\cU$ of subspaces obtained from a generator matrix $G$ of $\cC$ is such that every hyperplane of $\PG(hk-1, q)$ contains at most $n-d$ elements of $\cU$
and some hyperplane contains exactly $n-d$ elements of $\cU$. 
\end{theorem}

Note that different choices of the generator matrix $G$ yield different representatives of the same multisets of subspaces. Precisely, if $G'$ is another generator matrix for $\C$, then $G' = MG$ for some $M \in \GL(hk,q)$, and the multisets of subspaces defined by $G$ and $G'$ lie in the same orbit under $\PGL(hk,q)$. The following holds. 
\begin{theorem} [see \textnormal{\cite[Result 3.3]{adriaensen2023additive}}]
    There is a one to one correspondence between equivalence classes of $\F_q$-linear $(n,q^{kh},d)_{q^h}$ codes and orbits of multisets $\cU$ formed by $n$ subspaces of $\PG(hk-1, q)$ of projective dimension at most $h-1$, such that every hyperplane contains at most $n-d$ elements of $\cU$ and some hyperplane contains exactly $n-d$ elements of $\cU$. 
\end{theorem}

\noindent For the case $h=1$, the correspondence reduces to the classical connection between linear Hamming-metric codes and multiset of points in a projective space. \\

\noindent We now turn our attention to MDS codes.  
It is a classical result that linear MDS $(n,q^k)_q$ codes correspond to arcs in $\PG(k-1,q)$. Using the geometric description of additive codes, this correspondence can be extended to provide the geometric counterpart of additive MDS codes.

\begin{theorem} \label{th:MDSconnectionpseudo-arc}
    Let $\cC$ be an additive $(n, q^{kh}, d)_{q^h}$ code which is linear over $\F_q$. Then the set of subspaces obtained from a generator matrix $G$ of $\cC$ is a pseudo-arc in $\PG(hk-1, q)$ if and only if $\cC$ is an MDS code.
\end{theorem}

\subsection{The pseudo-arc \texorpdfstring{$\mathcal{P}_{h,k,q}$}{Phkq} and additive MDS codes} \label{sec:correspondenceadditive}

In view of Theorem \ref{th:MDSconnectionpseudo-arc}, the pseudo-arc $\mathcal{P}_{h,k,q}$ constructed in \Cref{th:explicitconstructionpseudoarc} correspond to additive MDS codes.  
In this subsection, we show that these additive MDS codes coincide precisely with the family of additive MDS codes recently introduced in \cite{neri2025skew}.  
Exploiting the geometric properties established for these pseudo-arcs, we then prove that the corresponding codes are not equivalent to any linear MDS code.\\

\noindent We start by recalling the additive MDS codes introduced in \cite{neri2025skew}.

\begin{definition}
Let $\boldsymbol{\alpha} = (\alpha_1,\ldots,\alpha_n)$, where $\alpha_1,\ldots,\alpha_n$ are distinct elements of $\Lambda_{h,q}$.  
For any integer $1 \leq k < n$, define the code
\begin{equation} \label{eq:sublinearanalogue}
\mathcal{S}_{h,k,q}(\boldsymbol{\alpha})
   := \bigl\{\, (f(\alpha_1),\ldots,f(\alpha_n)) : f(x) \in \F_q[x],\ \deg(f) < hk \,\bigr\} 
   \subseteq \F_{q^h}^n.
\end{equation}
\end{definition}
\noindent The codes $\mathcal{S}_{h,k,q}(\boldsymbol{\alpha})$ can be regarded as the \emph{sublinear analogue} of Reed--Solomon codes.  
Indeed, they are $\F_q$-linear MDS codes over $\F_{q^h}$, and can be obtained as subcodes of classical Reed--Solomon codes over $\F_{q^h}$ of dimension $hk$, evaluated on subsets of $\Lambda_{h,q}$. As proved in \cite[Theorem 5.17]{neri2025skew}, the code $\mathcal{S}_{h,k,q}(\boldsymbol{\alpha})$ is an $\F_q$-linear $(n,q^{hk},n-k+1\,)_{q^h}$
code, and hence an $\F_q$-linear MDS code in $\F_{q^h}^{n}$.  
The maximum length occurs when $n = |\Lambda_{h,q}|$. Note that a generator matrix of $\mathcal{S}_{h,k,q}(\boldsymbol{\alpha})$ is given by
\begin{equation} \label{eq:generatormatrixS}
G_{h,k,q}(\boldsymbol{\alpha})=
\begin{bmatrix}
1 & 1 & \cdots & 1\\
\alpha_1 & \alpha_2 & \cdots & \alpha_n\\
\alpha_1^{2} & \alpha_2^{2} & \cdots & \alpha_n^{2}\\
\vdots & \vdots & \ddots & \vdots\\
\alpha_1^{\,hk-1} & \alpha_2^{\,hk-1} & \cdots & \alpha_n^{\,hk-1}
\end{bmatrix} \in \F_{q^h}^{hk \times n}.
\end{equation}

\begin{proposition} \label{prop:descriptionhsystemS}
Let $\boldsymbol{\alpha} = (\alpha_1,\ldots,\alpha_n)$, where $\alpha_1,\ldots,\alpha_n$ are the distinct elements of $\Lambda_{h,q}$.  
For any integer $1 \leq k < n$, the set of subspaces obtained from the generator matrix $G_{k,h}(\boldsymbol{\alpha})$
coincides with the pseudo-arc $\mathcal{P}_{h,k,q}$ as in \eqref{eq:pseudopkhq}.
\end{proposition}
\begin{proof}
A generator matrix for $\mathcal{S}_{h,k,q}(\boldsymbol{\alpha})$ is $G = G_{h,k,q}(\boldsymbol{\alpha})$ as in \eqref{eq:generatormatrixS}.  
Let $\omega$ be a normal element of the extension $\F_{q^h}/\F_q$, and set $\boldsymbol{\omega} = (\omega, \omega^q, \ldots, \omega^{q^{h-1}})^{\top}$.  
The $j$-th column of $G$ can be written as
\begin{equation} \label{eq:sfolderjcolumns}
\begin{pmatrix}
1 \\
\alpha_j \\
\vdots \\
\alpha_j^{hk-1}
\end{pmatrix}
= G_j \boldsymbol{\omega},
\end{equation}
for some $G_j \in \F_q^{hk \times h}$. Keeping the notation of \Cref{sec:explicitconstruction}, recall that
\[
X^{*}(P_{\alpha_j}) = \langle P_{\alpha_j}, P_{\alpha_j}^{\Psi}, \ldots, P_{\alpha_j}^{\Psi^{h-1}} \rangle_{\F_{q^h}},
\]
where $P_{\alpha_j}$ is as in \eqref{eq:formimaginarypoints}
and note that
\[
P_{\alpha_j}^{\Psi^i} = \langle (1, \alpha_j^{q^i}, (\alpha_j^2)^{q^i}, \ldots, (\alpha_j^{hk-1})^{q^i}) \rangle_{\F_{q^h}}.
\]
If $G_j^1,\ldots,G_j^h$ denote the columns of $G_j$, by \eqref{eq:sfolderjcolumns}, we have
\[
P_{\alpha_j} \in \langle (G_j^1)^{\top},\ldots,(G_j^h)^{\top} \rangle_{\F_{q^h}}.
\]
Moreover, for every $i$,
\[
\begin{pmatrix}
1 \\
\alpha_j^{q^i} \\
(\alpha_j^2)^{q^i} \\
\vdots \\
(\alpha_j^{hk-1})^{q^i}
\end{pmatrix}
= G_j \begin{pmatrix} \omega^{q^i} \\  \omega^{q^{i+1}} \\ \vdots \\ \omega^{q^{h-1}} \\ \omega \\ \vdots \\ \omega^{q^{i-1}}\end{pmatrix},
\]
so that $P_{\alpha_j}^{\Psi^i} \in \langle (G_j^1)^{\top},\ldots,(G_j^h)^{\top} \rangle_{\F_{q^h}}$ as well. Therefore,
\[
X(P_{\alpha_j}) = X^{*}(P_{\alpha_j}) \cap \Sigma = \PG(U_j,\F_q),
\]
where $U_j$ is the $\F_q$-column space generated by the columns of $G_j$ and so a pseudo-arc associated with $\mathcal{S}_{h,k,q}(\boldsymbol{\alpha})$ is $\cP_{h,k,q}$ provided in \Cref{th:explicitconstructionpseudoarc}.
\end{proof}

The question of determining whether the codes $\mathcal{S}_{h,k,q}(\boldsymbol{\alpha})$ are equivalent to linear MDS codes remained open in \cite{neri2025skew}.  
In what follows, we answer this question in the negative by adopting a geometric approach.  
A geometric criterion for determining when an additive code is $\F_q$-equivalent to a linear code is recalled below.

\begin{proposition}[see \textnormal{\cite[Lemma 3.7]{adriaensen2023additive}}]
\label{prop:geometriccondnonequiv}
An $\F_q$-linear $(n,q^{hk},d)_{q^h}$ code $\cC$ is $\F_q$-equivalent to a linear code in $\F_{q^h}^n$ if and only if the set of subspaces obtained from a generator matrix $G$ of $\cC$ consists of $(h-1)$-dimensional projective subspaces contained in a Desarguesian $(h-1)$-spread of $\PG(hk-1,q)$.
\end{proposition}

Since, by \Cref{th:ourarcisnondesarguesian}, the pseudo-arc $\mathcal{P}_{h,k,q}$ is non-Desarguesian, we obtain the following result.

\begin{proposition} \label{prop:thecodeisnotlinear}
Assume that $q \ge hk+1$ and let $\boldsymbol{\alpha}=(\alpha_1,\ldots,\alpha_n)$, where
$\alpha_1,\ldots,\alpha_n$ are all the distinct elements of $\Lambda_{h,q}$.
For any integer $1 \le k < n$, the code
$
\mathcal{S}_{h,k,q}(\boldsymbol{\alpha}) \subseteq \F_{q^h}^n,
$
defined in \eqref{eq:sublinearanalogue}, is not $\F_q$-equivalent to any linear MDS code in
$\F_{q^h}^n$.
\end{proposition}

\begin{proof}
By \Cref{prop:descriptionhsystemS}, 
the set of subspaces obtained from the generator matrix $G_{h,k,q}(\boldsymbol{\alpha})$ is the pseudo-arc $\mathcal{P}_{h,k,q}$.
By Theorem \ref{th:ourarcisnondesarguesian}, the pseudo-arc $\mathcal{P}_{h,k,q}$ is
non-Desarguesian.
Therefore, by Proposition \ref{prop:geometriccondnonequiv}, the assertion follows
immediately.
\end{proof}
\subsection{Extending the code \texorpdfstring{$\mathcal{S}_{h,k,q}(\boldsymbol{\alpha})$}{Shkq} } \label{sec:codeextension}

As shown in \Cref{sec:extending}, the pseudo-arc $\mathcal{P}_{h,k,q}$ can be further extended to a larger pseudo-arc.
An immediate consequence on the coding-theoretic side is that the corresponding codes
$\mathcal{S}_{h,k,q}(\boldsymbol{\alpha})$ can be extended to longer additive MDS codes.
In this section, we provide an explicit description of these extended codes.

\begin{definition}
Let $\boldsymbol{\alpha} = (\alpha_1,\ldots,\alpha_n)$, where $\alpha_1,\ldots,\alpha_n$ are distinct elements of $\Lambda_{h,q}$, and let $\mathbf{t}=(t_1,\ldots,t_m)$, where $t_1,\ldots,t_m$ are distinct elements of $\F_q$.
Let $\omega$ be a normal element of $\F_{q^h}$ over $\F_q$.
For any integer $1 \le k < n$, define the code
\begin{equation} \label{eq:extendedcode}
\begin{aligned}
\mathcal{S}_{h,k,q}(\boldsymbol{\alpha},\mathbf{t},\infty)
   := \bigl\{\,(&f(\alpha_1),\ldots,f(\alpha_n),
   \sum_{i=0}^{h-1} f^{(i)}(t_1)\,\omega^{q^i}, \ldots,
   \sum_{i=0}^{h-1} f^{(i)}(t_m)\,\omega^{q^i}, f(\overline{\infty})) \\
   & :
   f(x)\in\F_q[x],\ \deg(f)<hk \,\bigr\}
   \subseteq \F_{q^h}^{\,n+m+1}.
\end{aligned}
\end{equation}
where, if $f(x)=\sum_{j=0}^{hk-1} f_j x^j$, we define
$
f(\overline{\infty}) := \sum_{i=0}^{h-1} f_{hk-h+i}\,\omega^{q^i}$.
\end{definition}

The code $\mathcal{S}_{h,k,q}(\boldsymbol{\alpha},\mathbf{t},\infty)$ extends
$\mathcal{S}_{h,k,q}(\boldsymbol{\alpha})$ by adding evaluation points corresponding
to derivatives and to the point at infinity.
We prove that the resulting code is still an additive MDS code by showing that it
is precisely the code associated with the extended pseudo-arc
$\mathcal{P}_{h,k,q} \cup \mathcal{O}_{h-1}$ as defined in \Cref{th:extension_pseudoarc}.

\begin{theorem}
    The code $\mathcal{S}_{h,k,q}(\boldsymbol{\alpha},\mathbf{t},\infty) \subseteq \F_{q^h}^{n+m+1}$ defined as in \eqref{eq:extendedcode} is an additive MDS code. 
\end{theorem}

\begin{proof}
By considering as an $\F_q$-basis of the space of polynomials in $\F_q[x]$ of degree less
than $hk$ the set $
\{1,x,x^2,\ldots,x^{hk-1}\},$
we can explicitly describe a generator matrix of
$\mathcal{S}_{h,k,q}(\boldsymbol{\alpha},\mathbf{t},\infty)$. Set  $\boldsymbol{\omega} = (\omega, \omega^q, \ldots, \omega^{q^{h-1}})^{\top}$.
More precisely, a generator matrix is given by
\[
G = \bigl(\, G_{h,k,q}(\boldsymbol{\alpha}) \;\big|\; H_1 \;\big|\; \cdots \;\big|\; H_m \;\big|\; H_{\infty} \,\bigr),
\]
where $G_{h,k,q}(\boldsymbol{\alpha})$ is as in \eqref{eq:generatormatrixS},
$H_i = A_{t_i}^{\top}\boldsymbol{\omega}$ for $i=1,\ldots,m$, and
$H_{\infty} = A_{\infty}^{\top}\boldsymbol{\omega}$.
Here $A_{t_i}$ and $A_{\infty}$ are defined in
\eqref{eq:basisosculating} and \eqref{eq:basisosculatinginfty}, respectively. Hence, the set of subspaces obtained from the generator matrix $G$ of $\mathcal{S}_{h,k,q}(\boldsymbol{\alpha},\mathbf{t},\infty)$ is
$
\mathcal{P}_{h,k,q} \cup \mathcal{O}_{h-1}
$. The assertion follows by \Cref{th:extension_pseudoarc} and \Cref{th:MDSconnectionpseudo-arc}. 
\end{proof}

We remark that the maximum length of the code $\mathcal{S}_{h,k,q}(\boldsymbol{\alpha},\mathbf{t},\infty)$ is
\[
|\Lambda_{h,q}| + q + 1,
\]
and that, if $q \ge hk-1$, it is not $\F_q$-equivalent to any linear MDS code in
$\F_{q^h}^n$, as immediate consequence of Proposition \ref{prop:thecodeisnotlinear}.

\subsection{Further comments}

For small parameters, additive MDS codes that are not equivalent to linear MDS codes have been classified in \cite[Section 6]{ball2023additive}. In particular, the authors have proved that the longest $(n, 4^6, n-2)_{4^2}$ additive MDS code that is not equivalent
to a linear MDS code has parameters $(11, 4^6
,9)_{4^2}$. Moreover, the corresponding pseudo-arc $\cP$ is unique up to a collineation of $\PG(5, 4)$ and can be described as follows. Let $w$ be a primitive element of $\F_{4^2}$ such that $w^4 = w+1$. Then $e = w^5$ is a primitive element of $\F_{4}$ satisfying $e^2 = e+1$ and $\cP$ consists of the lines:
\begin{align*}
    & \ell_1 = \langle (1,0,0,0,0,0), (0,1,0,0,0,0) \rangle_{\F_4}, && \ell_2 = \langle (1,0,e,e,e,1), (0,1,e,1,1,1) \rangle_{\F_4}, \\
    & \ell_3 = \langle (1,0,1,0,1,0), (0,1,0,1,0,1) \rangle_{\F_4}, && \ell_{4} = \langle (1,0,0,e,1,e), (0,1,e,e^2,e,0) \rangle_{\F_4}, \\
    & \ell_5 = \langle (1,0,e,e^2,0,1), (0,1,e^2,e^2,1,e^2) \rangle_{\F_4}, 
    && \ell_6 = \langle (0,0,1,0,0,0), (0,0,0,1,0,0) \rangle_{\F_4}, \\
    & \ell_7 = \langle (0,0,0,0,1,0), (0,0,0,0,0,1) \rangle_{\F_4}, && \ell_8 = \langle (1,0,0,1,0,e), (0,1,1,e^2,e,e^2) \rangle_{\F_4}, \\
    & \ell_9 = \langle (1,0,1,e^2,e^2,1), (0,1,1,1,0,e) \rangle_{\F_4}, && \ell_{10} = \langle (1,0,0,1,0,e), (0,1,1,e^2,e,e^2) \rangle_{\F_4}, \\
    & \ell_{11} = \langle (1,0,e,0,e,e), (0,1,1,e,e^2,e) \rangle_{\F_4}. 
\end{align*}

Actually, the pseudo-arc $\cP$ can be described as in Theorem \ref{th:extension_pseudoarc}. Indeed, embed $\PG(5, 4)$ in $\PG(5, 4^2)$ and consider the set $\cN$ formed by the points 
\begin{align*}
    & L_1 = \langle (1,0,0,0,0,0)\rangle_{\F_{16}}, && L_2 = \langle (1,1,0,e^2,e^2,0) \rangle_{\F_{16}}, \\
    & L_3 = \langle (1,e,1,e,1,e) \rangle_{\F_{16}}, 
    && L_4 = \langle (1,e^2,1,0,0,e) \rangle_{\F_{16}}, \\
    & L_5 = \langle (0,1,e^2,e^2,1,e^2) \rangle_{\F_{16}}, && L_6 = \langle (0,0,1,w^4,0,0) \rangle_{\F_{16}}, \\
     & L_6^{\Psi} = \langle (0,0,1,w^6,0,0) \rangle_{\F_{16}}, && L_7 = \langle (0,0,0,0,1,w^3) \rangle_{\F_{16}}, \\  & L_7^{\Psi} = \langle (0,0,0,0,1,w^{12}) \rangle_{\F_{16}}, &&
     L_8 = \langle (1,w^7,w^8,w^5,w^{14},w^{12}) \rangle_{\F_{16}}, \\ & L_8^{\Psi} = \langle (1,w^{13},w^2,w^5,w^{11},w^3) \rangle_{\F_{16}}, &&  L_9 = \langle (1,w^2,w^8,w^4,w^{10},w^9) \rangle_{\F_{16}}, \\
     & L_9^{\Psi} = \langle (1,w^8,w^2,w,w^{10},w^6) \rangle_{\F_{16}}, && L_{10} = \langle (1,w^{11},w^{11},w^{13},w,w^9) \rangle_{\F_{16}}, \\  & L_{10}^{\Psi} = \langle (1,w^{14},w^{14},w^7,w^4,w^6) \rangle_{\F_{16}}, 
    && L_{11} = \langle (1,w,w^2,w^6,w^3,w^9) \rangle_{\F_{16}}, \\ & L_{11}^{\Psi} = \langle (1,w^4,w^8,w^9,w^12,w^6)\rangle_{\F_{16}}. && 
\end{align*}
Here $\Psi$ is the involutory collineation of $\PG(5, 4^2)$ fixing pointwise $\PG(5, 4)$.
The projectivity induced by the matrix (acting on the left) 
\begin{align*}
\begin{pmatrix}
    1 & e & 0 & 1 & 1 & e^2 \\
    e & 1 & 1 & 1 & e^2 & 0 \\
    e^2 & 1 & e & 0 & 1 & 1 \\
    1 & e^2 & 1 & e & 0 & 1 \\
    e & e^2 & 1 & e^2 & e & e^2 \\
    e^2 & e & 0 & e & e & 1 
\end{pmatrix}
\end{align*}
maps $\cN$ to the normal rational curve $\cN_{6, 4^2}$ consisting of the points
\begin{align*}
    \left\{\langle (1,t,t^2,t^3,t^4,t^5) \rangle_{\F_{16}} \mid t \in \F_{16}\right\} \cup \left\{\langle (0,0,0,0,0,1)\rangle_{\F_{16}}\right\}. 
\end{align*}
Hence $\cN$ is a normal rational curve of $\PG(5, 4^2)$. Moreover, the line $\ell_i$ is the intersection of the line $\langle L_i, L_i^{\Psi}\rangle_{\F_{16}}$ with $\PG(5, 4)$, if $i = 6, \dots, 11$; whereas the line $\ell_i$ is the intersection of the line tangent to $\cN$ at the point $L_i$ with $\PG(5, 4)$ in the case when $i = 1,2,3,4,5$.

\bigskip
{\footnotesize
\noindent\textit{Acknowledgments.}
The research was supported by the Italian National Group for Algebraic and Geometric Structures and their Applications (GNSAGA--INdAM).}

\bibliographystyle{abbrv}
\bibliography{biblio}

\noindent Francesco Pavese,\\
Dipartimento di Meccanica, Matematica e Management, \\
Politecnico di Bari, \\
Via Orabona 4, \\
70125 Bari, Italy \\
E-mail: francesco.pavese@poliba.it\\

\noindent Paolo Santonastaso,\\
Dipartimento di Meccanica, Matematica e Management, \\
Politecnico di Bari, \\
Via Orabona 4, \\
70125 Bari, Italy \\
E-mail: paolo.santonastaso@poliba.it

\end{document}